\documentclass[oneside]{amsart} 
\usepackage{paralist}
\usepackage{graphics} 
\usepackage{epsfig} 
\usepackage{xspace}
\usepackage{marginnote} 

\usepackage{array}

\usepackage{appendix}
\usepackage[english]{babel} 
\usepackage[T1]{fontenc} 
\usepackage[utf8]{inputenc} 
\usepackage{geometry}
\usepackage{graphicx}
\usepackage{tabto}
\usepackage{amsmath}
\usepackage{calligra}
\usepackage{amsthm}
\usepackage{bbm}
\usepackage{cancel}
\usepackage{mathrsfs}
\usepackage{mathtools}
\usepackage[algo2e]{algorithm2e}
\usepackage{version}
\usepackage{algorithm}
\usepackage{arydshln}
\usepackage[math]{cellspace}

\usepackage{multicol}
\usepackage{listings}

\usepackage{matlab-prettifier}
\usepackage{bm}
\usepackage{subfig}

\usepackage{tikz}

\usepackage{amssymb}
\newcommand{\numberset}{\mathbb}

\newcommand{\R}{\numberset{R}}

\newcommand{\K}{\numberset{K}}
\newcommand{\A}{\numberset{A}}
\newcommand{\BB}{\numberset{B}}
\newcommand{\C}{\numberset{C}}
\newcommand{\VV}{\numberset{V}}

\newcommand\II{\mathcal{I}}

\DeclareFontFamily{U}{matha}{\hyphenchar\font45}
\DeclareFontShape{U}{matha}{m}{n}{
	<-6> matha5 <6-7> matha6 <7-8> matha7
	<8-9> matha8 <9-10> matha9
	<10-12> matha10 <12-> matha12
}{}
\DeclareSymbolFont{matha}{U}{matha}{m}{n}
\DeclareFontFamily{U}{mathx}{\hyphenchar\font45}
\DeclareFontShape{U}{mathx}{m}{n}{
	<-6> mathx5 <6-7> mathx6 <7-8> mathx7
	<8-9> mathx8 <9-10> mathx9
	<10-12> mathx10 <12-> mathx12
}{}
\DeclareSymbolFont{mathx}{U}{mathx}{m}{n}
\DeclareMathDelimiter{\vvvert} {0}{matha}{"7E}{mathx}{"17}%
\newcommand{\norm}[1]{\left\lVert#1\right\rVert}
\newcommand{\abs}[1]{\left\lvert#1\right\rvert}
\newcommand{\normiii}[1]{\left\vvvert#1\right\vvvert}

\newtheorem{teo}{Theorem}
\newtheorem{defi}{Definition}
\newtheorem{pro}{Problem}
\newtheorem{lem}{Lemma}
\newtheorem{prop}{Proposition}
\newtheorem{rem}{Remark}

\newtheorem{assump}{Assumption}
\newtheorem{corollary}{Corollary}

\let\div\undefined\DeclareMathOperator{\div}{div} 
\let\curl\undefined\DeclareMathOperator{\curl}{curl} 
\DeclareMathOperator*{\Grad}{\boldsymbol\nabla}
\DeclareMathOperator*{\Grads}{{\underline{\boldsymbol{\varepsilon}}}}
\DeclareMathOperator{\grads}{\boldsymbol\nabla_s}
\newcommand{\derivative}[2]{\frac{\partial #1}{\partial #2}}
\newcommand\Huo{\mathbf{H}^1_0(\Omega)}
\newcommand\Ldo{L^2_0(\Omega)} 
\newcommand\Hub{\mathbf{H}^1(\B)} 
\newcommand\HuOm{\mathbf{H}^1(\Omega)} 
\newcommand\Hubd{\Hub'} 
\newcommand\LdBd{\mathbf{L}^2(\B)} 
\newcommand\LdOd{\mathbf{L}^2(\Omega)} 
\newcommand\Winftyd{\mathbf{W}^{1,\infty}(\B)}
\newcommand\Of{\Omega^f} 
\newcommand\Os{\Omega^s} 
\newcommand\B{\mathcal B} 
\renewcommand\u{\mathbf{u}} 
\renewcommand\a{\mathbf{a}} 
\renewcommand\v{\mathbf{v}} 
\renewcommand\d{\mathbf{d}} 
\renewcommand\c{\mathbf{c}} 
\newcommand\f{\mathbf{f}} 
\newcommand\g{\mathbf{g}} 
\newcommand\w{\mathbf{w}} 
\newcommand\x{\mathbf{x}} 
\newcommand\s{\mathbf{s}} 
\renewcommand\S{\mathbf{S}} 
\renewcommand\H{\mathbf{H}}

\newcommand\V{\mathcal{V}} 
\newcommand\U{\mathcal{U}} 
\newcommand\W{\mathcal{W}} 

\newcommand\X{\mathbf{X}} 
\newcommand\Xbar{\overline{\X}} 
\newcommand\Xcirc{\X^{\mathrm{m}}_h}
\newcommand\Y{\mathbf{Y}} 
\newcommand\Ytilde{\tilde{\mathbf{Y}}} 
\newcommand\pwg{{\varphi}}

\newcommand\LL{{\boldsymbol{\Lambda}}} 
\newcommand\Vline{\mathbf{V}} 
\newcommand\Sline{\S} 

\newcommand\llambda{\boldsymbol\lambda} 
\newcommand\mmu{{\boldsymbol\mu}} 
\newcommand\ds{\mathrm{d}\s} 
\newcommand\dx{\mathrm{d}\x} 
\newcommand\dt{\Delta t} 
\newcommand\T{\mathcal{T}} 
\newcommand\quadnode{\mathbf{p}}
\newcommand\quadweigth{\omega}
\newcommand\Lcal{\mathcal{L}}
\newcommand\Pcal{\mathcal{P}}
\newcommand\Vdiv{\Vline_{0,h}} 
\newcommand\tri{\mathrm{T}} 
\newcommand\trifj{\tri_{\mathrm{f},j}} 
\newcommand\tris{\tri_\mathrm{s}} 
\newcommand\reftri{\widehat{\tri}} 
\newcommand\projLdue{\Pi^0} 

\newcommand{\Amatr}{\mathsf{A}}
\newcommand{\Bmatr}{\mathsf{B}}
\newcommand{\Cmatr}{\mathsf{C}}

\newcommand\llambdaI{\Pi^0\llambda}

\newcommand{\crhs}[2]{\c_h(#1,#2)}
\newcommand{\crhsL}[2]{\c_{\ell,h}(#1,#2)}
\newcommand{\crhsZ}[2]{\c_{0,h}(#1,#2)}
\newcommand{\crhsU}[2]{\c_{1,h}(#1,#2)}

\usepackage{listings}
\renewcommand\lg{\begin{color}{black}}
	\newcommand\gl{\end{color}}
\newcommand\blue{\begin{color}{black}}
	\newcommand\noblue{\end{color}}

\newcommand\DB{\begin{color}{black}}
	\newcommand\BD{\end{color}}

\newcommand\fc{\begin{color}{magenta}}
	\newcommand\cf{\end{color}}

\begin{document}
	
	\title[]{Quadrature error estimates on non--matching grids\\in a fictitious domain framework\\for fluid--structure interaction problems}
	
	\author[Daniele Boffi]{Daniele Boffi}
	\address{Computer, electrical and mathematical sciences and engineering division, King Abdullah University of Science and Technology, Thuwal 23955, Saudi Arabia and Dipartimento di Matematica \textquoteleft F. Casorati\textquoteright, Universit\`a degli Studi di Pavia, via Ferrata 5, 27100, Pavia, Italy}
	\email{daniele.boffi@kaust.edu.sa}
	\urladdr{kaust.edu.sa/en/study/faculty/daniele-boffi}
	
	\author{Fabio Credali}
	\address{Computer, electrical and mathematical sciences and engineering division, King Abdullah University of Science and Technology, Thuwal 23955, Saudi Arabia. Former affiliation: Istituto di Matematica Applicata e Tecnologie Informatiche \textquoteleft E. Magenes\textquoteright, Consiglio Nazionale delle Ricerche, via Ferrata 5, 27100, Pavia, Italy }
	\email{fabio.credali@kaust.edu.sa -- fabio.credali@imati.cnr.it}
	
	\author{Lucia Gastaldi}
	\address{Dipartimento di Ingegneria Civile, Architettura, Territorio, Ambiente e di Matematica, Universit\`a degli Studi di Brescia, via Branze 43, 25123, Brescia, Italy}
	\email{lucia.gastaldi@unibs.it}
	\urladdr{lucia-gastaldi.unibs.it}
	
	
	\begin{abstract}
		We consider a fictitious domain formulation for
fluid-structure interaction problems based on a distributed Lagrange
multiplier to couple the fluid and solid behaviors. How to deal with the
coupling term is crucial since the construction of the associated finite
element matrix requires the integration of functions defined over non-matching
grids: the exact computation can be performed by intersecting the involved
meshes, whereas an approximate coupling matrix can be evaluated on the
original meshes by introducing a quadrature error. The purpose of this paper is twofold: we prove that the discrete problem is well-posed also when the coupling term is constructed in approximate way and we discuss quadrature error estimates over non-matching grids.
		
		\
		
	\end{abstract}
	
	\maketitle
	
	\section{Introduction}
	
	A wide range of physical phenomena generating interest within the scientific community can be classified as fluid-structure interaction problems. Several applications can be found, for instance, in biology and medicine, as well as in structural engineering. Due to the complexity of the equations governing this kind of systems, which usually involve nonlinear terms and complex geometries, several mathematical models have been developed during the last decades.
	
	Mathematical approaches to fluid-structure interaction problems are
	usually divided into two categories. In \textit{boundary fitted approaches},
	fluid and solid are discretized by two meshes sharing the interface. The solid
	evolution is usually described in Lagrangian framework, whereas the fluid grid
	evolves and deforms around the solid one during the simulation. The Arbitrary
	Lagrangian Eulerian (ALE)
	formulation~\cite{hirt1974arbitrary,donea2004arbitrary,hughes1981lagrangian}
	belongs to this category. The family of \textit{boundary unfitted approaches},
	on the other hand, considers meshes which are not fitting with the interface.
	Several approaches fall in this family: we mention, for instance, the level set formulation~\cite{chang1996level}, the Nitsche--XFEM method~\cite{alauzet2016nitsche,burman2014unfitted}, the fictitious domain approach introduced in~\cite{glowinski2001fictitious,glowinski1997lagrange}, the CutFEM method~\cite{burman2015cutfem}, the shifted boundary method \cite{atallah2021analysis} and the immersed boundary-conformal isogeometric method \cite{wei2021immersed}. 
	In our case, the fluid and solid are discretized by two completely independent
	meshes and then the solid discretization is in some way superimposed to the
	fluid one. We will be more precise in the description of our model.
	In case of \textit{fitted} approaches, the kinematic constraints combining fluid dynamics and solid evolution are automatically satisfied thanks to the construction of the method itself, but mesh distortion may cause stability issues. On the other hand, \textit{unfitted} approaches overcome the presence of distorted meshes, but they require additional effort to impose kinematic constraints. This can be done, for instance, by introducing Lagrange multipliers or stabilization terms on the mesh elements close to the interface.
	
	In this paper, we focus on the approach introduced in~\cite{2015}.
	This method is an evolution of Peskin's immersed boundary
	method~\cite{peskin1972flow,peskin2002immersed} where, first, finite element are
	used for the approximation of the partial differential
	equations~\cite{boffi2003finite} and, then, a
	Lagrange multiplier formulation is considered in the spirit of the
	fictitious
	domain method.
	Fluid and solid domain are discretized by two independent meshes. More
	precisely, evolution and deformation of the immersed solid body are studied by
	Lagrangian description: the equations are defined on a fixed reference domain,
	which is mapped, at each time step, into the actual position of the solid. The
	fluid dynamics is described in Eulerian coordinates on a fixed mesh, which is
	extended to the region occupied by the structure.
	
	A distributed Lagrange multiplier and a suitable bilinear form are
	used to impose the kinematic constraint at variational level: this condition relates the solid material velocity, defined on the fictitiously extended fluid domain, with the Lagrangian map, defined on the solid reference domain. The additional term appearing in the variational formulation is also known as \textit{coupling term}. In the finite element framework, the associated coupling (or interface) matrix is obtained by integrating the product of solid and fluid basis functions over the solid elements.
	In particular, the fluid basis functions are composed with the solid
	mapping at the previous time step. This poses the issue of integrating functions defined on two different non-matching grids: the accuracy of this procedure not only depends on the precision of the considered quadrature rule, but also on the algorithm we employ to compute the reciprocal position between the two meshes.
	Several works have been focused on similar aspects: among others, we quote, for example, the interpolation technique between unfitted grids discussed in~\cite{fromm2023interpolation}, the efficient three dimensional CutFEM approach presented in~\cite{massing2013efficient}, the variational transfer method discussed and implemented in~\cite{krause2016parallel} and, finally, in the case of codimension one interface problems, the mortar element approach introduced in~\cite{maday2002influence}.
	
	In~\cite{boffi2022interface}, we presented and discussed two possible integration techniques that can be adopted to assemble the finite element coupling matrix. The coupling term can be assembled \textit{exactly} if we  compute the intersection between the fluid mesh and the mapped solid one and use a quadrature formula which is exact for the polynomials involved. With this procedure, we construct a finer mesh for the solid domain, which is used just to implement a composite quadrature rule able to take fully into account the interaction between fluid and solid elements. Since this algorithm may be computationally demanding (the coupling term is assembled at each time step), we discuss in this paper an \textit{inexact} approach by integrating directly over the solid elements.
	In this case, the fluid elements interacting with the solid are determined by the position of the mapped quadrature nodes, so that an additional source of error is introduced. The purpose of this paper is twofold: we prove the well-posedness of the discrete problem when the coupling term is constructed in approximate way and we provide estimates for the quadrature error. Our theoretical results are confirmed also at numerical level.
	
	The paper is organized as follows. After recalling the functional analysis notation in Section~\ref{sec:notation}, we discuss the problem setting in Section~\ref{sec:problem_setting}. We then summarize in Section~\ref{sec:analysis} the existing theoretical results concerning the well-posedness of both the continuous and the discrete problem. In Section~\ref{sec:interface}, we discuss the assembly techniques for the coupling term and, in Section~\ref{sec:abstract-results}, we prove a general theorem showing the effect of numerical integration. Quadrature error estimates for the coupling term are then presented and proved in Section~\ref{sec:quad_error}, while in Section~\ref{sec:infsup} we discuss the well-posedness of the discrete problem in case of approximate coupling. In the last Section~\ref{sec:numerical_tests}, we report some numerical tests confirming our theoretical results.
	
	\section{Notation}\label{sec:notation}
	Before starting our discussion, we recall some useful functional analysis notation.
	
	Let us consider an open bounded domain $D\subset\R^d$. We denote by $L^2(D)$ the space of square integrable functions on $D$ with inner product $(\cdot,\cdot)_D$. $L^2_0(D)$ is the subspace of $L^2(D)$ of zero mean valued functions.
	
	For standard Sobolev spaces, we use the notation $W^{s,p}(D)$ where
	$s\in\R$ is the differentiability and $p\in[1,\infty]$ the integrability
	exponent, respectively. When $p=2$, we denote the space by $H^s(D)$ with norm
	$\norm{\cdot}_{s,D}$ and seminorm $\abs{\cdot}_{s,D}$. In particular, if $s$
	is a fractional exponent, we are going to use the following expression of
	the fractional seminorm
	\begin{equation*}
		\abs{u}_{s,D}^2 = \int_D \int_D \frac{\abs{u(x)-u(y)}^2}{\abs{x-y}^{d+2s}} \,dxdy
	\end{equation*}
	for a given function $u\in H^s(D)$.
	We denote by $H^1_0(D)$ the subset of $H^1(D)$ of functions with zero trace on $\partial D$.
	
	Vector and tensor valued functions are denoted by boldface letters, as well as the Sobolev spaces they belong to.
	
	Finally, $\mathscr{L}(M_1,M_2)$ denotes the space of linear and continuous functionals between two functional spaces $M_1$ and $M_2$. 
	
	\section{Problem setting}\label{sec:problem_setting}
	
	In this paper we deal with the formulation introduced and studied
	in~\cite{2015,stat} and in subsequent papers. We refer the interested reader to
	that reference for the derivation of the model, since here we are essentially
	interested in the treatment of the coupling term. For this reason, we
	consider a simplified model which represents the steady state formulation
	of a fluid-structure interaction problem. Our results are representative on what we should expect in more general situations.
	
	We consider an open bounded domain $\Omega\subset\R^d$ ($d=2,3$) with
	Lipschitz continuous boundary. This domain is the union of two disjoint
	regions, denoted by $\Of$ and $\Os$, occupied by the fluid and the solid
	respectively. In our study, we assume that the solid boundary $\partial\Os$
	cannot touch $\partial\Omega$, that is
	$\partial\Os\cap\partial\Omega=\emptyset$, and we focus on the case with solid
	of codimension zero, even if the method can deal with codimension one problems
	(see e.g.~\cite{2015,annese}) and can be extended to codimension two case
	(see e.g.~\cite{ALZETTA,heltaizunino}).
	
	We adopt the Eulerian description for the fluid, denoting by $\x$ the
	associated variable and the behavior of the solid is described using the
	Lagrangian approach. We introduce a reference domain $\B\subset\R^d$ for the
	solid, associated with variable $\s$, which is mapped into the current
	position of the body through the action of a map $\X$, which is one of the
	unknowns of our problem, so that $\x=\X(\s)\in\Os$ and $\Os=\X(\B)$, see
	Figure~\ref{fig:geo}.
	After time semi-discretization, we consider the mapping $\Xbar$ at the
	previous time step.
	The other unknowns of our system are the velocity $\u$ and pressure $p$ of the fluid which are extended, in the spirit of the fictitious domain, to
	the entire domain~$\Omega$.
	
	\begin{figure}
		\centering
		\begin{tikzpicture}
			\draw[fill=cyan!20, thick] (0,0) -- (4,0) -- (4,4) -- (0,4) -- (0,0);
			\draw (4.5,3.7) node {$\Omega$};
			\draw (0.5,3.7) node {$\Of$};	
			
			\draw (-4.5,3.7) node {$\B$};
			\coordinate (c1) at (-4.5,0.5);
			\coordinate (c2) at (-3.5,0.7);
			\coordinate (c3) at (-2,1.8);
			\coordinate (c4) at (-2,3);
			\coordinate (c5) at (-3,3.5);
			\coordinate (c8) at (-4.5,3);
			\coordinate (c7) at (-4.7,2);
			\coordinate (c6) at (-4.8,1);
			\draw[fill = yellow!30, thick] plot [smooth cycle] coordinates {(c1) (c2) (c3) (c4) (c5) (c8) (c7) (c6)};			
			
			\coordinate (c1) at (1.5,0.5);
			\coordinate (c2) at (2.5,0.7);
			\coordinate (c3) at (3,1.8);
			\coordinate (c4) at (3,3);
			\coordinate (c5) at (2,3.5);
			\coordinate (c8) at (1.5,3);
			\coordinate (c7) at (1,2);
			\coordinate (c6) at (0.6,1);
			\draw[fill = yellow!30, thick] plot [smooth cycle] coordinates {(c1) (c2) (c3) (c4) (c5) (c8) (c7) (c6)};
			\draw (2.2,3) node {$\Os$};
			
			\draw (-3.2,2.3) node {$\s$};
			\draw (2.1,2) node {$\x$};
			
			\draw (-1.1,2.4) node {$\X(\s)$};
			
			\draw[->,thick,violet] (-3,2.3)  .. controls (0,2)  ..  (1.9,2);
		\end{tikzpicture}
		
		\
		
		\caption{Geometric configuration of the problem. A Lagrangian point
			$\s$ of the solid reference domain $\B$ is mapped into the actual position of the body $\Os$ through the map $\X$. The domain $\Omega$ acts like a container.}
		\label{fig:geo}
		
	\end{figure}
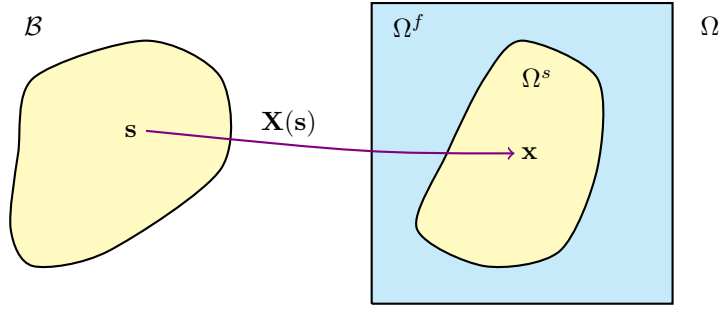
	
	The last unknown is a Lagrange multiplier $\llambda$ defined on the reference
	solid domain $\B$ which enforces the kinematic condition
	\[
	\u^s(\x,t) = \derivative{\X(\s,t)}{t},
	\]
	where $\x=\X(\s,t)$ and $\u^s$ is the restriction of $\u$ to $\Os$.
	
	We can reformulate the above condition, discretely in time, as
	\begin{equation}
		\u(\Xbar) - \frac{\X}{\dt} = -\frac{\Xbar}{\dt}.
	\end{equation}
	For the sake of simplicity and without losing generality, we use the condition
	\begin{equation}\label{eq:kinematic}
		\u(\Xbar)-\gamma\X=\d,
	\end{equation}
	where we fix the constant $\gamma$ equal to one to avoid heavier notation.
	
	The multiplier $\llambda$ is chosen in a space $\LL$ defined on $\B$.
	We consider a continuous bilinear form $$\c:\LL\times\Hub\to\R$$
	with the property
	\begin{equation}\label{eq:def_prop}
		\c(\mmu,\Y) = 0 \qquad \forall \mmu\in\LL \qquad \implies \qquad \Y=0.
	\end{equation}
	We report here two possible choices of $\LL$ and $\c$: let $\LL_0=\Hubd$, the dual space of $\Hub$, then
	\begin{equation}
		\label{eq:c_l2}
		\c_0(\mmu,\Y) = \langle\mmu,\Y\rangle \quad \forall\mmu\in \Hubd,\Y\in \Hub;
	\end{equation}
	alternatively, let $\LL_1=\Hub$, then $\c$ can be taken as the scalar product
	in $\Hub$
	\begin{equation}
		\label{eq:c_h1}
		\c_1(\mmu,\Y) = (\grads\mmu,\grads\Y)_\B + (\mmu,\Y)_\B\quad \forall\mmu,\Y\in \Hub.
	\end{equation}
	In the remainder, we use generally the notation $\c$ and we specify
	$(\LL_0,\c_0)$ or $(\LL_1,\c_1)$ when necessary.
	
	The kinematic condition~\eqref{eq:kinematic} can then be written as
	\[
	\c(\mmu,\u(\Xbar)-\X) = \c(\mmu,\d) \qquad \forall\mmu\in\LL
	\]
	by exploiting the property~\eqref{eq:def_prop} of the bilinear form $\c$.
	
	Our final variational formulation can then be written by introducing suitable
	bilinear forms taking into account the fluid and solid models as
	in~\cite{stat}:
	\begin{equation}\label{eq:forme}
		\begin{aligned}
			&\a_f(\u,\v)=\alpha(\u,\v)_\Omega+(\nu\Grads(\u),\Grads(\v))_\Omega&& \forall\u,\v\in\Huo\\
			&\a_s(\X,\Y) = \beta(\X,\Y)_\B+\kappa (\grads\X,\grads\Y)_\B && \forall\X,\Y\in\Hub.
		\end{aligned}
	\end{equation}
	\DB In~\cite{stat} we considered a model involving the fully nonlinear
	Navier--Stokes equations. Here, for simplicity we present only the linear case
	corresponding to the Stokes problem.\BD
	
	We are now in a position to write the problem we are dealing with in our
	paper.
	
	\begin{pro}
		\label{pro:stationary_general}
		Let $\Xbar\in\Winftyd$ be invertible with Lipschitz inverse. Given
		$\f\in\LdOd$, $\g\in \LdBd$, and $\d\in\Hub$, find $(\u,p)\in\Huo\times\Ldo$,
		$\X\in\Hub$, and $\llambda\in\LL$ such that
		\begin{subequations}
			\begin{align}
				&\a_f(\u,\v)-(\div\v,p)_\Omega+\c(\llambda,\v(\Xbar))=(\f,\v)_\Omega&&
				\forall\v\in\Huo\\
				&(\div\u,q)_\Omega=0&&\forall q\in \Ldo\\
				&\a_s(\X,\Y)-\c(\llambda,\Y)=(\g,\Y)_\B&&\forall\Y\in\Hub\\
				&\c(\mmu,\u(\Xbar)-\X)=\c(\mmu,\d)&&\forall\mmu\in\LL
			\end{align}
		\end{subequations}
	\end{pro}
	
	\section{Analysis of the problem}\label{sec:analysis}
	
	In this section, we recall existing results regarding the well-posedness of Problem~\ref{pro:stationary_general} and its finite element discretization.
	The problem has been already studied in terms of well-posedness and stability: since this is the starting point for our new results, we  recall the main features of the theory presented in \cite{stat}. We introduce two new bilinear forms ${\A:\VV\times\VV\rightarrow\R}$ and ${\BB:\VV\times\Ldo\rightarrow\R}$ defined as
	\begin{equation}
		\begin{aligned}
			&\A(\U,\V) = \a_f(\u,\v) + \a_s(\X,\Y) + \c(\llambda,\v(\Xbar)-\Y) - \c(\mmu,\u(\Xbar)-\X)\\
			&\BB(\V,q) = \DB-(\div\v,q)_\Omega\BD
		\end{aligned}
	\end{equation}
	where $\VV$ is the product space $\VV = \Huo\times\Hub\times\LL$ endowed with the norm
	\begin{equation}
		\normiii{\V} = \big(\norm{\v}^2_{1,\Omega} + \norm{\Y}^2_{1,\B} + \norm{\mmu}^2_{\LL}\big)^{1/2},
	\end{equation}
	with generic element $\V=(\v,\Y,\mmu)$.
	
	In this way, Problem~\ref{pro:stationary_general} can be rewritten as follows: let $\Xbar\in\Winftyd$ be invertible with Lipschitz inverse, given $\f\in \LdOd$, $\g\in \LdBd$ and $\d\in\Hub$, find $(\U,p)\in\VV\times\Ldo$ such that
	\begin{equation}\label{eq:mix_pb}
		\begin{aligned}
			&\A(\U,\V) + \BB(\V,p) = (\f,\v)_\Omega + (\g,\Y)_\B - \c(\mmu,\d)&&\forall\V\in\VV\\
			&\BB(\U,q)=0&&\forall q\in\Ldo.
		\end{aligned}
	\end{equation}
	We introduce  the solution operator ${\Lcal:\VV\times\Ldo\rightarrow\VV^\prime\times\Ldo^\prime}$ associated to the left hand side of~\eqref{eq:mix_pb}
	\begin{equation}
		\langle \Lcal(\U,p),(\V,q) \rangle = \A(\U,\V) + \BB(\V,p) +\BB(\U,q).
	\end{equation}
	Hence, the solution of~\eqref{eq:mix_pb} is characterized by
	\begin{equation}\label{eq:ell_operator}
		\Lcal(\U,p) = (\f,\g,\d,\mathbf{0}).
	\end{equation}
	
	This problem is well-posed since the following inf-sup conditions are satisfied \cite{stat}.
	\begin{prop}
		There exist two positive constants $\eta$ and $\theta$ such that
		\begin{equation}
			\inf_{q\in\Ldo}\sup_{\V\in\VV}\frac{\BB(\V,q)}{\normiii{\V}\norm{q}_{0,\Omega}}\ge\eta,
			\qquad
			\inf_{\U\in\K_{\BB}}\sup_{\V\in\K_{\BB}}\frac{\A(\U,\V)}{\normiii{\U}\normiii{\V}}\ge\theta
		\end{equation}
		where $\K_\BB=\{\V\in\VV:\BB(\V,q)=0\quad\forall q\in\Ldo\}$. Consequently, thanks to the theory in \cite{mixedFEM}, there exists a unique solution of Problem~\ref{pro:stationary_general}.
	\end{prop}
	
	We now introduce the finite element discretization
	of~\eqref{eq:mix_pb}. Let us consider a mesh $\T_h^\Omega$ in $\Omega$ with
	meshsize $h_\Omega$ and a mesh $\T_h^\B$ with size $h_\B$ in the solid reference domain $\B$. We introduce four finite element spaces $\Vline_h\subset\Huo$, $Q_h\subset \Ldo$, $\Sline_h\subset \Hub$, $\LL_h\subset\LL$, taking care that $\Vline_h$ and $Q_h$ satisfy the discrete inf-sup conditions for the Stokes problem. In our analysis, we assume $\Sline_h=\LL_h$, even if $\Sline_h$ and $\LL_h$ can be either equal or different spaces, especially in the case $\c=\c_0$ (see \cite{najwa} for more general cases). 
	In the next sections, \lg we consider triangular meshes. Fluid
	velocity and pressure are discretized by the Bercovier--Pironneau element
	$\Pcal_1$-iso-$\Pcal_2/\Pcal_1$, which was introduced and analyzed in \cite{bercovier}. Piecewise linear elements are adopted for the mapping and the Lagrange multiplier. \gl
	\lg More precisely, we set
	\begin{equation}\label{eq:element}
		\begin{aligned}
			\Vline_h &= \{ \v\in\HuOm: \v_{\mid
				\tri}\in[\Pcal_1(\tri)]^2 \quad \forall \tri\in\T_{h/2}^\Omega \} \\
			Q_h &= \{q\in \Ldo\cap H^1(\Omega): q_{\mid \tri}\in \Pcal_1(\tri)\quad\forall \tri\in\T_{h}^\Omega\}\\
			\S_h &= \{\w\in\Hub:\w_{\mid \tri}\in [\Pcal_1(\tri)]^2\quad\forall \tri\in\mathcal{T}_{h}^\B\}\\
			\LL_h&=\{\mmu\in\Hub:\mmu_{\mid \tri}\in [\Pcal_1(\tri)]^2\quad\forall \tri\in\mathcal{T}_{h}^\B\}.
		\end{aligned}
	\end{equation}
	Let us remark that velocity and pressure are defined on two different
	(nested) meshes: indeed, $\T_{h/2}^\Omega$ is obtained by connecting the middle points of each triangle of $\T_h^\Omega$.
	\gl
	
	Finally, we recall that if $\c=\c_0$, given in \eqref{eq:c_l2}, and if
	$\mmu\in\mathbf{L}^1_{loc}(\B)$, then it is possible to identify the duality
	pairing with the scalar product in $\LdBd$. Since $\LL_h$ is a finite element
	subspace of $\LdBd$, we have then
	\begin{equation}\label{eq:c_l2_inner}
		\c_0(\mmu_h,\Y_h) = (\mmu_h,\Y_h)_\B \quad \forall\mmu_h\in\LL_h,\forall\Y_h\in\Sline_h.
	\end{equation}
	
	Problem \ref{pro:stationary_general} can be written in discrete form as follows.
	\begin{pro}
		\label{pro:discrte_stationary_general}
		Let $\Xbar\in\Winftyd$ be invertible with Lipschitz inverse. Given $\f\in \LdOd$, $\g\in \LdBd$ and $\d\in \Hub$, find $(\u_h,p_h)\in\Vline_h \times Q_h$, $\X_h\in \Sline_h$ and $\llambda_h\in \LL_h$, such that
		\begin{subequations}
			\begin{align}
				&\a_f(\u_h,\v_h)-(\div\v_h,p_h)_\Omega+\c(\llambda_h,\v_h(\Xbar))=(\f,\v_h)_\Omega && 
				\forall\v_h\in\Vline_h\\
				&(\div\u_h,q_h)_\Omega=0&&\forall q_h\in Q_h\\
				& \a_s(\X_h,\Y_h)-\c(\llambda_h,\Y_h)=(\g,\Y_h)_\B&&\forall\Y_h\in \Sline_h \\
				& \c(\mmu_h,\u_h(\Xbar)-\X_h)=\c(\mmu_h,\d)&&\forall\mmu_h\in \LL_h.
			\end{align}
		\end{subequations}
	\end{pro}
	
	As for the continuous problem, the well-posedness of Problem~\ref{pro:discrte_stationary_general} is studied using the theory in~\cite{mixedFEM}. We define the discrete product space ${\VV_h = \Vline_h\times\Sline_h\times\LL_h}$, still endowed with the norm $\normiii{\cdot}$, so that the problem can be reformulated by using $\A$ and $\BB$. It reads: 
	let $\Xbar\in\Winftyd$ be invertible with Lipschitz inverse, given $\f\in \LdOd$, $\g\in \LdBd$ and $\d\in\Hub$, find $(\U_h,p_h)\in\VV_h\times Q_h$ such that
	\begin{equation}\label{eq:pb_disc_sp}
		\begin{aligned}
			&\A(\U_h,\V_h) + \BB(\V_h,p_h) = (\f,\v_h)_\Omega + (\g,\Y_h)_\B - \c(\mmu_h,\d)&&\forall\V_h\in\VV_h\\
			&\BB(\U_h,q_h)=0&&\forall q_h\in Q_h.
		\end{aligned}
	\end{equation}
	The discrete counterpart of the operator $\Lcal$ reads
	\begin{equation}
		\Lcal_h:\VV_h\times Q_h \longrightarrow \VV_h^\prime\times Q_h^\prime \quad\text{so that}\quad
		\Lcal_h(\U_h,p_h) = (\f,\g,\d,\mathbf{0}).
	\end{equation}
	Problem~\ref{pro:discrte_stationary_general} is well-posed since the discrete inf-sup conditions hold true~\cite{stat}.
	\begin{prop}\label{prop:infsup_discreteproblem}
		There exist two positive constants $\tilde{\eta}$ and $\tilde{\theta}$, independent of h, such that
		\begin{equation}\label{eq:infsup_discrete}
			\inf_{q_h\in Q_h}\sup_{\V_h\in\VV_h}\frac{\BB(\V_h,q_h)}{\normiii{\V_h}\norm{q_h}_{0,\Omega}}\ge\tilde{\eta},
			\qquad
			\inf_{\U_h\in\K_{\BB,h}}\sup_{\V_h\in\K_{\BB,h}}\frac{\A(\U_h,\V_h)}{\normiii{\U_h}\normiii{\V_h}}\ge\tilde{\theta}
		\end{equation}
		where
		\begin{equation}\label{eq:disc_ker}
			\K_{\BB,h}=\{\V_h\in\VV_h:\BB(\V_h,q_h)=0\quad\forall q_h\in Q_h\}.
		\end{equation}
		Therefore, there exists a unique solution of Problem~\ref{pro:discrte_stationary_general}.
	\end{prop}
	
	Hence the optimal convergence theorem follows.
	\begin{teo}
		\label{theo:brezzi}
		Let $\Vline_h$ and $Q_h$ satisfy the usual compatibility
		conditions for the Stokes problem. If $(\u,p,\X,\llambda)$ and
		$(\u_h,p_h,\X_h,\llambda_h)$ denote respectively the solution for the
		continuous and the discrete problem \lg(\eqref{eq:mix_pb}
		and~\eqref{eq:pb_disc_sp}, respectively)\gl, then the following error estimate holds true
		\begin{equation*}
			\begin{aligned}
				&\norm{\u-\u_h}_{1,\Omega} + \norm{p-p_h }_{0,\Omega}+\norm{\X-\X_h }_{1,\B}+\norm{ \llambda-\llambda_h}_{\LL}\\
				&\hspace{1mm}\leq C\big( \inf_{\v_h\in\Vline_h}\norm{\u-\v_h}_{1,\Omega} + \inf_{q_h\in Q_h}\norm{p-q_h}_{0,\Omega} + \inf_{\Y_h\in\Sline_h}\norm{\X-\Y_h}_{1,\B} + 
				\inf_{\mmu_h\in\LL_h}\norm{\llambda-\mmu_h}_{\LL} \big).
			\end{aligned}
		\end{equation*}
	\end{teo}
	
	In contrast with popular unfitted methods, which can achieve higher
	order convergence when the solution is piecewise smooth, our method
	converges with a rate depending on the global regularity of the
	solution.
	In the wide literature of unfitted finite elements, each method has
	advantages and disadvantages. We refer the interested reader to~\cite{comparison}
	for a comparison of our method and other popular unfitted schemes, such
	as the CutFEM method, in the case of interface problems.
	
	\section{Computational aspects}\label{sec:interface}
	
	In this section, we briefly describe some computational aspects regarding the discrete problem in~\eqref{eq:pb_disc_sp}, with particular focus to the coupling term. We now restrict our discussion to $d=2$ with triangular meshes for both solid and fluid domain. 
	
	First, let us write the problem in matrix form
	\begin{equation}\renewcommand{\arraystretch}{1.5}
		\label{eq:matrix}
		\left[\begin{array}{@{}ccc|c@{}}
			\Amatr_f & 0 & \Cmatr_f^\top & \Bmatr^\top\\ 
			0  & \Amatr_s & -\Cmatr_s^\top &0 \\
			\Cmatr_f &  -\Cmatr_s & 0 & 0\\
			\hline
			\Bmatr & 0 & 0 & 0 \\
		\end{array}\right]
		\left[ \begin{array}{c}
			\u_h\\
			\X_h\\
			\llambda_h\\
			\hline
			p_h\\
		\end{array}\right]=
		\left[ \begin{array}{c}
			\f\\
			\g\\
			\d\\
			\hline
			\mathbf{0}
		\end{array}\right].
	\end{equation}\renewcommand{\arraystretch}{1.15}
	
	Notice that $\Amatr_f$, $\Bmatr$ act only on $\u_h$ and $p_h$, which
	are defined on the fluid mesh $\T^\Omega_h$, while $\Amatr_s$ and $\Cmatr_s$ act only on the solid variables $\X_h$ and $\llambda_h$ defined on $\T^\B_h$. All these matrices can be assembled in exact way provided that a sufficiently precise quadrature rule is employed. This fact is not anymore true if we look at $\Cmatr_f$, which aims at coupling the behavior of fluid and solid in the fictitious part of the fluid domain.
	%
	Possible techniques that can be adopted for this procedure have been
	presented and discussed with several numerical tests
	in~\cite{boffi2022interface}. In this section we only recall the main features
	of this computational procedure and, in the next sections, we investigate the
	quadrature error committed when the construction is not exact.
	
	The interface matrix originates from the form
	$\c(\mmu_h,\v_h(\Xbar))$, where we have to compute integrals on $\B$ involving $\mmu_h\in\LL_h$ and $\v_h\in\Vline_h$, that are discrete functions defined on two different meshes. Notice that the velocity-like function $\v_h$ is composed with the map $\Xbar$ to take into account the actual position of the solid body. To fix the ideas, we report in Figure~\ref{fig:mapping_tri} a simple example of a solid triangle immersed into the fluid mesh through $\Xbar$: the mismatch between the supports of fluid and mapped solid basis functions is evident.
	
	\begin{figure}
		\centering
		\subfloat[]{\includegraphics[width=0.55\textwidth]{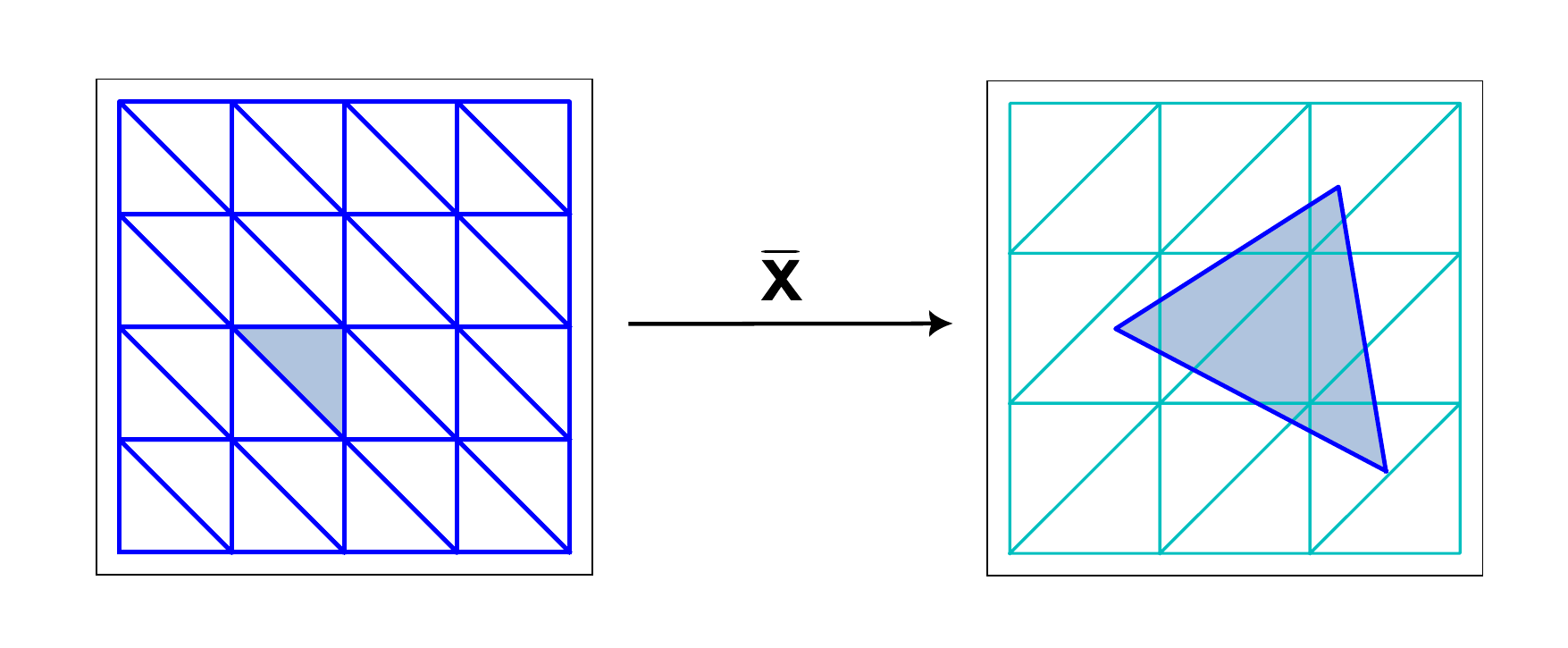}}\quad
		\subfloat[]{\includegraphics[width=0.31\textwidth]{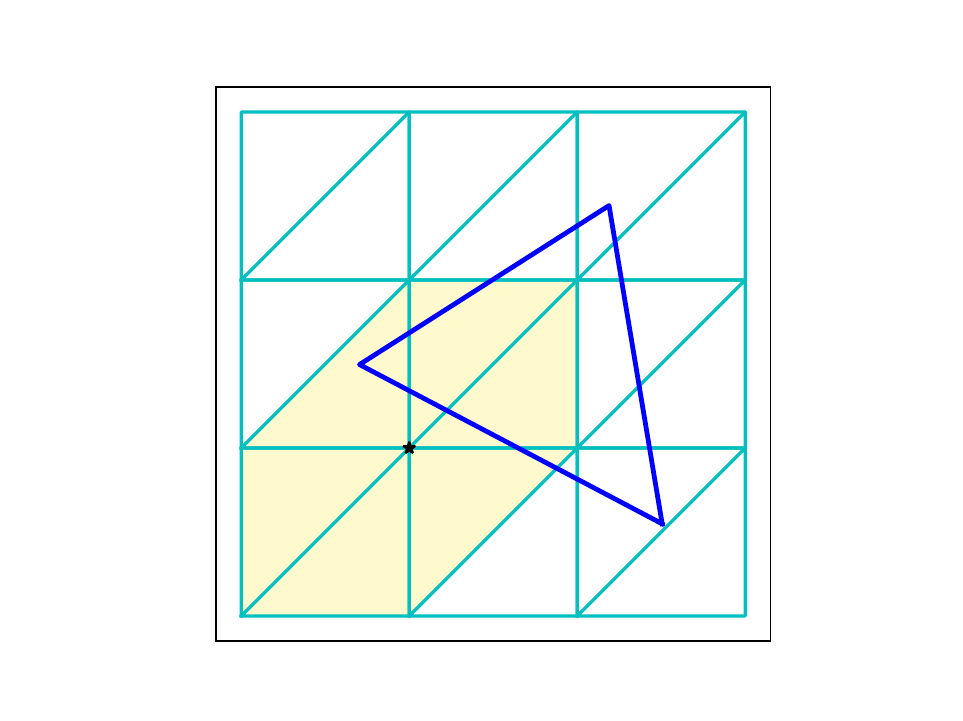}}
		\caption{On the left, mapping of a solid element into the fluid mesh. On the right, the yellow support of the fluid basis function associated with the starred node only partially matches the mapped solid triangle.}
		\label{fig:mapping_tri}
	\end{figure}
	
	In the following, we denote by $\{(\quadnode_k^0,\quadweigth_k^0)\}_{k=1}^{K_0}$ nodes and weights for a generic quadrature rule for $\LdBd$ inner product of functions, while, for the product of gradients, we use a rule with nodes and weights denoted by $\{(\quadnode_k^1,\quadweigth_k^1)\}_{k=1}^{K_1}$. Both formulas are defined on a generic triangle $\tri\in\T^\B_h$.
	
	The discrete counterparts of the two possible choices for the coupling term are the following:
	\begin{subequations}\label{eq:ch}
		\begin{equation}\label{eq:ch_l2}
			\c_0(\mmu_h,\v_h(\Xbar)) = \sum_{\tris\in\T^\B_h}\int_{\tris} \mmu_h\cdot \v_h(\Xbar)\,\ds,
		\end{equation}
		\begin{equation}\label{eq:ch_h1}
			\c_1(\mmu_h,\v_h(\Xbar)) 
			= \sum_{\tris\in\T^\B_h}\int_{\tris} \left(\mmu_h\cdot \v_h(\Xbar) + \grads\mmu_h:\grads\v_h(\Xbar)\right)\,\ds.
		\end{equation}
	\end{subequations}
	
	The exact computation of these terms can be carried out by making use
	of a composite quadrature rule, taking into account that $\v_h(\Xbar)$ is a piecewise polynomial in each $\tris\in\T^\B_h$. To this aim, one should compute the intersection between the fluid and the mapped solid mesh.
	
	Otherwise, we can proceed in an approximate way directly using a quadrature formula, without paying attention to the fact that $\v_h(\Xbar)$ is a piecewise polynomial in each $\tris\in\T_h^\B$. We have for \eqref{eq:ch}
	\begin{subequations}\label{eq:noint}
		\begin{equation}\label{eq:noint_l2}
			\c_{0,h}(\mmu_h,\v_h(\Xbar)) = \sum_{\tris\in\T^\B_h}\abs{\tris} \sum_{k=1}^{K_0} \quadweigth_k^0\,\mmu_h(\quadnode_k^0)\cdot \v_h(\Xbar(\quadnode_k^0)),
		\end{equation}
		\begin{equation}\label{eq:noint_h1}
			\begin{aligned}
				&\c_{1,h}(\mmu_h,\v_h(\Xbar)) \\
				&\quad= \sum_{\tris\in\T^\B_h}\abs{\tris}\left(  \sum_{k=1}^{K_0} \quadweigth_k^0\,\mmu_h(\quadnode_k^0)\cdot \v_h(\Xbar(\quadnode_k^0)) + \sum_{k=1}^{K_1} \quadweigth_k^1\,\grads\mmu_h(\quadnode_k^1):\grads\v_h(\Xbar(\quadnode_k^1))\right).
			\end{aligned}
		\end{equation}
	\end{subequations}
	
	In the following, the notation $\c_h(\cdot,\cdot)$ is used for both cases, as in the continuous setting. It is clear that the approach in \eqref{eq:noint} produces a quadrature error, which is the subject of our study.
	
	\section{The effect of numerical integration: abstract results}\label{sec:abstract-results}
	
	In this section, following \cite{brezzi74,mixedFEM}, we present a
	general result regarding the use of numerical integration to approximate the
	coupling terms and the right hand sides of our problem. We denote by
	$(\cdot,\cdot)_{h,D}$ the discrete $L^2(D)$ inner product obtained using
	numerical integration. Moreover, for all $\mmu_h\in\LL_h$ and $\Y\in\Hub$,
	$\crhs{\mmu_h}{\Y}$ stands for the approximation of $\c(\mmu_h,\Y)$ by means of a quadrature rule. 
	We recall that we are assuming $\S_h=\LL_h$. Problem~\ref{pro:discrte_stationary_general} with inexact integration reads as follow.	
	
	\begin{pro}
		\label{pro:approx_stationary_general}
		Let $\Xbar\in\Winftyd$ be invertible with Lipschitz inverse. Given $\f\in \LdOd$, $\g\in \LdBd$ and $\d\in\Hub$, find $(\u_h^\star,p_h^\star)\in\Vline_h \times Q_h$, $\X_h^\star\in \Sline_h$ and $\llambda_h^\star\in \LL_h$, such that
		\begin{subequations}
			\begin{align}
				&\a_{f}(\u_h^\star,\v_h)-(\div\v_h,p_h^\star)_\Omega+\c_h(\llambda_h^\star,\v_h(\Xbar))=(\f,\v_h)_{h,\Omega} && 
				\forall\v_h\in\Vline_h\\
				&\label{div_ex}(\div\u_h^\star,q_h)_\Omega=0&&\forall q_h\in Q_h\\
				& \a_{s}(\X_h^\star,\Y_h)-\c(\llambda_h^\star,\Y_h)=(\g,\Y_h)_{h,\B}&&\forall\Y_h\in \Sline_h \\
				& \c_h(\mmu_h,\u_h^\star(\Xbar))-\c(\mmu_h,\X_h^\star)=\crhs{\mmu_h}{\d}&&\forall\mmu_h\in \LL_h.
			\end{align}
		\end{subequations}
	\end{pro}
	
	Similarly to the continuous and discrete case, we now introduce a new bilinear form ${\A_h:\VV_h\times\VV_h\rightarrow\R}$ given by
	\begin{equation}
		\begin{aligned}
			&\A_h(\U_h,\V_h) = \a_{f}(\u_h,\v_h) + \a_{s}(\X_h,\Y_h)\\
			&\hspace{2cm} + \c_h(\llambda_h,\v_h(\Xbar)) - \c(\llambda_h,\Y_h) - \c_h(\mmu_h,\u_h(\Xbar)) + \c(\mmu_h,\X_h)\\
		\end{aligned}
	\end{equation}
	for all $\U_h,\,\V_h\in\VV_h$. We do not need to introduce $\BB_h$ since $\BB$ can be computed exactly.
	
	Replacing $\A$ with $\A_h$, we can rewrite~\eqref{eq:pb_disc_sp} in matrix form as follows
	\begin{equation}\renewcommand{\arraystretch}{1.5}
		\label{eq:matrix2}
		\left[\begin{array}{@{}ccc|c@{}}
			\Amatr_f & 0 & \Cmatr_{f,h}^\top & \Bmatr^\top\\ 
			0  & \Amatr_s & -\Cmatr_s^\top &0 \\
			\Cmatr_{f,h} &  -\Cmatr_s & 0 & 0\\
			\hline
			\Bmatr & 0 & 0 & 0 \\
		\end{array}\right]
		\left[ \begin{array}{c}
			\u_h^\star\\
			\X_h^\star\\
			\llambda_h^\star\\
			\hline
			p_h^\star\\
		\end{array}\right]=
		\left[ \begin{array}{c}
			\f_h\\
			\g_h\\
			\d_h\\
			\hline
			\mathbf{0}
		\end{array}\right],
	\end{equation}\renewcommand{\arraystretch}{1.15}
	where $\Cmatr_{f,h}$, $\f_h$, $\g_h$ and $\d_h$ refer to terms computed with inexact integration.
	
	Moreover, we define the operator $\Lcal_h^\star:\VV_h\times Q_h\rightarrow\VV_h^\prime\times Q_h^\prime$ by
	\begin{equation}
		\langle \Lcal_h^\star(\U_h,p_h),(\V_h,q_h) \rangle = \A_h(\U_h,\V_h) + \BB(\V_h,p_h) + \BB(\U_h,q_h)\quad\forall\V_h\in\VV_h, \forall q_h\in Q_h
	\end{equation}
	in such a way that $(\U_h^\star,p_h^\star)$ represents the solution of Problem~\ref{pro:approx_stationary_general}, that is ${\Lcal_h^\star(\U_h^\star,p_h^\star)=(\f_h,\g_h,\d_h,\mathbf{0})}$.
	
	Our goal is to measure the distance between the approximated solution $(\u_h^\star,p_h^\star,\X_h^\star,\llambda_h^\star)$ and the continuous one $(\u,p,\X,\llambda)$. 
	The error analysis relies again on inf-sup conditions similar to \eqref{eq:infsup_discrete}. Since the bilinear form $\BB$ is computed exactly, we assume that the discrete form $\A_h$ satisfied the inf-sup condition; the proof is postponed to Section~\ref{sec:infsup}. 
	\begin{assump}\label{assump:infsup_appr}
		$\A_h$ satisfies the inf-sup condition, that is there exists a positive constant $\theta^\star$ independent of h such that
		\begin{equation}
			\inf_{\U_h\in\K_{\BB,h}}\sup_{\V_h\in\K_{\BB,h}}\frac{\A_h(\U_h,\V_h)}{\normiii{\U_h}\normiii{\V_h}}\ge\theta^\star,
		\end{equation}
		where $\K_{\BB,h}$ was defined in~\eqref{eq:disc_ker}.
	\end{assump}
	
	Thanks to the results in~\cite{xu2003some}, Assumption~\ref{assump:infsup_appr} and the first inf-sup condition in~\eqref{eq:infsup_discrete}~\cite{brezzi74} imply the following inf-sup condition~\cite{babuvska1973finite} 
	\begin{equation}\label{eq:strang_1}
		(\normiii{\W_h}^2 + \norm{r_h}_{0,\Omega}^2)^{1/2} \le M \sup_{\substack{\V_h\in\VV_h\\q_h\in Q_h}} \frac{ \langle\Lcal_h^\star(\W_h,r_h),(\V_h,q_h)\rangle}{(\normiii{\V_h}^2+\norm{q_h}^2_{0,\Omega})^{1/2}}	
	\end{equation}
	for $\W_h\in\VV_h$ and $r_h\in Q_h$, with the constant $M$ depending on the inf-sup constants $\tilde{\eta}$ and $\theta^\star$. We observe that, in this case, $\langle\cdot,\cdot\rangle$ denotes the duality pairing between $\VV_h\times Q_h$ and $\VV_h^\prime\times Q_h^\prime$.
	
	Exploiting that $(\U_h,p_h)$ is solution to the discrete Problem~\ref{pro:discrte_stationary_general}, we get
	\begin{equation}\label{eq:strang_2}
		\begin{aligned}
			\langle\Lcal_h^\star(\U_h-\U_h^\star,&p_h-p_h^\star),(\V_h,q_h)\rangle
			= \langle(\Lcal_h^\star-\Lcal_h)(\U_h,p_h),(\V_h,q_h)\rangle \\
			&\hspace{-0.35cm}+ [(\f,\v_h)_\Omega - (\f,\v_h)_{h,\Omega}]  + [(\g,\Y_h)_\B - (\g,\Y_h)_{h,\B}] + [\c(\d,\mmu_h)-\crhs{\d}{\mmu_h}].
		\end{aligned}
	\end{equation}
	The definitions of $\Lcal_h$ and $\Lcal_h^\star$ imply
	\begin{equation}\label{eq:strang_3}
		\begin{aligned}
			\langle (\Lcal_h^\star-\Lcal_h)(\U_h,p_h),(\V_h,q_h)\rangle &=
			\c_h(\llambda_h,\v_h(\Xbar))-\c(\llambda_h,\v_h(\Xbar))\\ &\quad+\c_h(\mmu_h,\u_h(\Xbar))-\c(\mmu_h,\u_h(\Xbar)).
		\end{aligned}
	\end{equation}
	Therefore, we obtain
	\begin{equation}\label{eq:strang_4}
		\begin{aligned}
			&\left(\normiii{\U_h-\U_h^\star}^2 + \norm{p_h-p_h^\star}^2_{0,\Omega}\right)^{1/2} \\
			&\quad\le M\sup_{\substack{\V_h\in\VV_h\\q_h\in Q_h}}\biggl(
			\frac{|\c(\mmu_h,\u_h(\Xbar))-\c_h(\mmu_h,\u_h(\Xbar)| + |\c(\llambda_h,\v_h(\Xbar))-\c_h(\llambda_h,\v_h(\Xbar))|}{(\normiii{\V_h}^2+\norm{q_h}^2_{0,\Omega})^{1/2}}\\
			&\qquad\quad\quad+ \frac{ \abs{(\f,\v_h)_\Omega - (\f,\v_h)_{h,\Omega}}+\abs{(\g,\Y_h)_\B - (\g,\Y_h)_{h,\B}} + \abs{\c(\d,\mmu_h)-\crhs{\d}{\mmu_h}} }{(\normiii{\V_h}^2+\norm{q_h}^2_{0,\Omega})^{1/2}}\biggl).
		\end{aligned}
	\end{equation}
	
	Before presenting the final estimate, we introduce the following assumption which will be then proved in the next section.
	\begin{assump}\label{assump:quad_er}
		There exist positive functions $\rho_{\ell}(h)$, $\ell=0,1$, tending to zero as $h$ goes to zero, such that for all $\mmu_h\in\LL_h$ and $\v_h\in\Vline_h$ it holds
		\begin{equation*}
			|\c_\ell(\mmu_h,\v_h(\Xbar))-\c_{\ell,h}(\mmu_h,\v_h(\Xbar)| \le C \rho_\ell(h)\,\|\mmu_h\|_{\ell,\B}\|\v_h\|_{1,\Omega}
		\end{equation*}
		with $\rho_{\ell}(h)$ depending on the definition of $\c$.
	\end{assump}
	It turns out that if $\f$, $\g$, and $\d$ are sufficiently regular, then the second term on the right hand side of~\eqref{eq:strang_4} can be estimated by the classical theory provided that the quadrature rule is appropriately chosen (see, for instance, \cite{ciarlet2002finite}).
	\begin{lem}
		Let $s>d/2$ be an integer number. Let $\f\in\mathbf{H}^s(\Omega)$,
		$\g\in \mathbf{H}^s(\B)$ and consider a quadrature rule exact for polynomials
		of degree $2s-2$. Then it holds
		\begin{equation}\label{eq:rhs_fg}
			\begin{aligned}
				&\abs{(\f,\v_h)_\Omega - (\f,\v_h)_{h,\Omega}}\le Ch_\Omega^s\,|\f|_{s,\Omega}\|\v_h\|_{1,\Omega}&&\forall\v_h\in\Vline_h\\
				&\abs{(\g,\Y_h)_\B - (\g,\Y_h)_{h,\B}}\le Ch_\B^s\,|\g|_{s,\B}\|\Y_h\|_{1,\B}&&\forall\Y_h\in\Sline_h.\\
			\end{aligned}
		\end{equation}
		Moreover, given $\d\in\mathbf{H}^r(\B)$ with $r>d/2+\ell$,
		$\mmu_h$ piecewise linear function in $\B$, and a quadrature rule exact for
		polynomials of degree $2$, the following bound holds
		\begin{equation}\label{eq:rhs_d}
			|\c_\ell(\mmu_h,\d)-\crhsL{\mmu_h}{\d}|\le
			Ch_\B^{r-\ell}\DB|\d|_{r,\B}\BD\|\mmu_h\|_{\ell,\B}\qquad \forall\mmu_h\in\LL_h.
		\end{equation}
	\end{lem}
	
	\begin{proof}
		The estimates~\eqref{eq:rhs_fg} are standard results, see~\cite[Chap. 4, Sect. 4.1]{ciarlet2002finite}. Concerning~\eqref{eq:rhs_d}, we distinguish the cases $\ell=0,1$. 
		
		If $\ell=0$, $\c_0(\cdot,\cdot)=(\cdot,\cdot)_\B$ and $\crhsZ{\cdot}{\cdot}$ is the corresponding approximation with a quadrature rule exact for polynomials of degree $2$. For all $\mmu_h$, we have
		\begin{equation*}
			\c_0(\mmu_h,\d) - \crhsZ{\mmu_h}{\d} = \sum_{\tri\in\T^\B_h} \left( \int_\tri \mmu_h\cdot\d\,\ds - |\tri|\sum_{k=1}^K \quadweigth_k\mmu_h(\quadnode_k)\cdot\d(\quadnode_k) \right).
		\end{equation*}
		We estimate the contribution of each element $\tri$ as follows. Let $\d_I$ be the interpolant of $\d$; taking into account the degree of exactness of the quadrature rule, we have 
		\[
		\begin{aligned}
			&\int_\tri \mmu_h\cdot\d\,\ds - |\tri|\sum_{k=1}^K \quadweigth_k\mmu_h(\quadnode_k)\cdot\d(\quadnode_k)\\
			&\quad= \int_\tri \mmu_h\cdot(\d-\d_I)\,\ds + |\tri|\sum_{k=1}^K \quadweigth_k\mmu_h(\quadnode_k)\cdot(\d_I(\quadnode_k)-\d(\quadnode_k))\\
			&\DB\quad\le Ch_\B^r \|\mmu_h\|_{0,\tri}|\d|_{r,\tri}+
			\left(|\tri|\sum_{k=1}^K\quadweigth_k|\mmu_h(\quadnode_k)|^2\right)^{1/2}
			\left(|\tri|\sum_{k=1}^K\quadweigth_k|\d_I(\quadnode_k)-\d(\quadnode_k)|^2\right)^{1/2}\BD\\
			&\DB\quad\le Ch_\B^r \|\mmu_h\|_{0,\tri}|\d|_{r,\tri}+
			|\tri|^{1/2}\|\mmu_h\|_{0,\tri}\|\d_I-\d\|_{\infty,\tri}\BD\\
			&\DB\quad\le Ch_\B^r \|\mmu_h\|_{0,\tri}|\d|_{r,\tri}+
			C|\tri|^{1/2}\|\mmu_h\|_{0,\tri}|\tri|^{-1/2}h_\B^r|\d|_{r,\tri}\BD\\
			&\DB\quad\le Ch_\B^r\|\mmu_h\|_{0,\tri}|\d|_{r,\tri}.\BD
		\end{aligned}
		\]
		The last inequality has been obtained using standard interpolation estimates,
		the discrete Cauchy--Schwarz inequality, and the precision of the
		quadrature rule. This implies~\eqref{eq:rhs_d} for $\ell=0$.
		
		For $\ell=1$, $\c_1(\mmu_h,\d)=(\mmu_h,\d)_\B+(\grads\mmu_h,\grads\d)_\B$ and $\crhsU{\mmu_h}{\d}$ is its approximation by means of the same quadrature rule. Hence by using the same argument as before, we arrive at~\eqref{eq:rhs_d} for $\ell=1$.
	\end{proof}
	
	Hence, the final result reads as follows.
	\begin{teo}\label{theo:strang}
		Under Assumption \ref{assump:infsup_appr} and
		Assumption~\ref{assump:quad_er}, if $(\u,p,\X,\llambda)$ is the solution to
		Problem \ref{pro:stationary_general}, \DB $(\u_h,p_h,\X_h,\llambda_h)$ is the solution to
		Problem \ref{pro:discrte_stationary_general}, \BD and
		$(\u_h^\star,p_h^\star,\X_h^\star,\llambda_h^\star)$ is the solution to
		Problem \ref{pro:approx_stationary_general}, then the following error estimate
		holds for $\ell=0,1$
		\begin{equation*}
			\begin{aligned}
				&\norm{\u-\u_h^\star}_{1,\Omega} + \norm{p-p_h^\star}_{0,\Omega} + \norm{\X-\X_h^\star}_{1,\B} + \norm{\llambda-\llambda_h^\star}_{\LL} \\
				&\hspace{0.2cm}\le C\,\big(\inf_{\v_h\in\Vline_h}\norm{\u-\v_h}_{1,\Omega} + \inf_{q_h\in Q_h}\norm{p-q_h}_{0,\Omega} + \inf_{\Y_h\in\Sline_h}\norm{\X-\Y_h}_{1,\B} + \inf_{\mmu_h\in\LL_h}\norm{\llambda-\mmu_h}_{\LL} \\
				&\hspace{1.5cm}+ (h_\B^{min})^{\ell-1}\rho_\ell(h)\|\u_h\|_{1,\Omega}+\rho_{\ell}(h)\,\|\llambda_h\|_{\ell,\B}
				+ h_\Omega^s\,|\f|_{s,\Omega} + h_\B^s\,|\g|_{s,\B} +\,
				h_\B^{r-\ell}\|\d\|_{r,\B}\big).
			\end{aligned}
		\end{equation*}
	\end{teo}
	
	\section{Quadrature error for the coupling term}\label{sec:quad_error}
	
	The goal of this section is to estimate the quadrature error occurring
	between the exact bilinear form $\c$ and its numerical counterpart $\c_h$
	computed with approximate integration. More precisely, we want to check
	Assumption~\ref{assump:quad_er} and determine $\rho_{\ell}(h)$ in terms of $h_\Omega$ and $h_\B$.
	
	From now on, we consider the case $d=2$ with triangular
	meshes and we use the Bercovier--Pironneau element for the Stokes part of the
	system and continuous piecewise linear elements for both the solid unknown and the Lagrange multiplier (see~\eqref{eq:element}).
	
	We introduce the following notion of quadrature error functional over a
	generic element $\tri$~\cite{ciarlet2002finite}.
	
	\begin{defi} Given a generic function $f$ and a quadrature rule with nodes and weights $\{(\quadnode_k,\quadweigth_k)\}_{k=1}^K$, the quadrature error functional $E_\tri$ over $\tri$ is defined as the difference between the exact integral and the numerical one, i.e. 
		\begin{equation}
			E_\tri(f) = \int_\tri f(\x)\,\dx - \abs{\tri}\sum_{k=1}^K \quadweigth_k f(\quadnode_k).
		\end{equation}
	\end{defi}
	
	The analysis on the quadrature error we are going to perform will strongly depend on the function $\Xbar$. Therefore, we introduce the following assumption, which is reasonable for the FSI application.
	
	\begin{assump}\label{assum:xbar}	
		We assume that $\Xbar\in\Sline_h$. This implies that it is linear in $\tri$ for all $\tri\in\T_h^B$. Consequently, the composed function
		$\v_h(\Xbar)$ is continuous and piecewise linear on
		$\tri$.
	\end{assump} 
	
	The next lemma states a fundamental inequality for continuous and piecewise linear functions $\v_h(\Xbar)$ that will be used later on.
	
	\begin{lem}
		Let $\tri\in\T_h^\B$ be such that $\Xbar(\tri)$ is not included in an element of
		$\T_{h/2}^\Omega$ and let Assumption~\ref{assum:xbar} hold true. Then $\v_h(\Xbar)\in \H^{1+s}(\tri)$ for $0\le s<1/2$ and 
		\[
		\|\grads\v_h(\Xbar)\|_{s,\tri}\le \frac{C}{1-2s} \|\grads\v_h(\Xbar)\|_{0,\tri}.
		\]
		\label{le:duran}
	\end{lem}
	
	\begin{proof}
		Let $\pwg$ be one of the components of $\grads\v_h$. Since $\tri\in\T_h^\B$ is such
		that $\Xbar(\tri)$ is not included in an element of $\T_{h/2}^\Omega$, we can
		subdivide $\tri$ in polygons $P_j$ with $j=1,\dots,J$ so that
		$T=\bigcup_{j=1}^J P_j$ and $\pwg$ is constant on each $P_j$. We set $\pwg_j=\pwg|_{P_j}$. Since, 
		${H^1(P_j)\subset H^s(P_j)}$, for $0\le s<1/2$, we have that the extension by
		zero $\tilde\pwg_j$ of $\pwg_j\in H^s(P_j)$ belongs to $H^s(\tri)$ with the
		following bound, see~\cite[Chap.~1, Th.~11.4]{lions2012non} and~\cite[Sect.~3.1]{duran2020analysis} 
		\[
		\|\tilde\pwg_j\|_{s,\tri}\le \frac{C}{1-2s}\|\pwg_j\|_{s,P_j}.
		\]
		Then $\pwg=\sum_{j=1}^J\tilde\pwg_j$ belongs to $H^s(\tri)$ and since
		$\pwg_j$ is constant for $j=1,\dots,J$, we have
		\[
		\|\pwg\|_{s,\tri}\le \frac{C}{1-2s}
		\left(\sum_{j=1}^J\|\pwg_j\|^2_{s,P_j}\right)^{1/2}
		= \frac{C}{1-2s}
		\left(\sum_{j=1}^J\|\pwg_j\|^2_{0,P_j}\right)^{1/2}.
		\]
		Hence we have proved that $\grads\v_h\in \H^s(\tri)$ and
		\[
		\|\grads\v_h\|_{s,\tri}\le \frac{C}{1-2s} \|\grads\v_h\|_{0,\tri}.
		\]
		
	\end{proof}
	
	\DB If $\widehat T$ is a reference element corresponding to $T\in\T$, we
	denote by $\widehat\bullet$ quantities defined on $\widehat T$ associated with
	$\bullet$.\BD
	
	We have now all the tools we need to prove the error estimates for the two choices of $\c$.

	\begin{prop}\label{prop:l2_quad_error}
		Let $\Vline_h$ and $\LL_h$ be given by \eqref{eq:element}. Let us assume that $\Xbar$ satisfies Assumption~\ref{assum:xbar}.
		Given a quadrature rule $\{(\quadnode_k^0,\quadweigth_k^0)\}_{k=1}^{K_0}$ exact for quadratic polynomials, i.e. such that
		\begin{equation}\label{eq:quadr}
			\widehat{E}(\widehat{f}) = 0 \quad\forall\widehat{f}\in\Pcal_2(\reftri),
		\end{equation}
		the following estimate holds true 
		\begin{equation}\label{eq:l2_quad_error}
			\abs{E_\B(\mmu_h\cdot\v_h(\Xbar))}
			\le
			C h_\B^{3/2}|\log h_\B^{min}|
			\norm{\mmu_h}_{0,\B}\DB\abs{\v_h}_{1,\Omega}\BD
			\quad\forall\mmu_h\in\LL_h,\forall\v_h\in\Vline_h,
		\end{equation}
		where $E_\B$ is the sum of $E_\tri(\mmu_h\cdot\v_h(\Xbar))$ for all $\tri\in\T_h^\B$ and  $h_\B^{min}=\min_{\tri\in\T_h^\B} h_\tri$. 
	\end{prop}
	
	\begin{proof}
		We start proving a local estimate in a single element $\tri$ of the solid mesh
		$\T_h^\B$: we notice that, thanks to~\eqref{eq:quadr}, if
		$\DB\Xbar(\tri)\BD$ is included in
		an element of the velocity mesh $\T_{h/2}^\Omega$, then the error is zero. For this reason, we consider the
		following two subsets of $\T_h^\B$
		\begin{equation}\label{eq:solid_mesh_partition}
			\begin{aligned}
				&\T_{h,1}^\B = \{\tri\in\T_h^\B:\DB\Xbar(\tri)\BD\text{ is included in an element of }
				\T_{h/2}^\Omega\}
				\\
				&\T_{h,2}^\B = \T_h^\B \setminus \T_{h,1}^\B.
			\end{aligned}
		\end{equation}
		Hence, $\v_h(\Xbar)$ is a piecewise linear
		polynomial on $\tri\in\T_{h,2}^\B$. 
		Our aim is to find a bound for
		\begin{equation}
			E_\tri(\mmu_h\cdot\v_h(\Xbar)) = \int_\tri \mmu_h\cdot\v_h(\Xbar)\,\ds 
			- \abs{\tri}\sum_{k=1}^{K_0}\quadweigth_k^0\mmu_h(\quadnode_k^0)\cdot\v_h(\Xbar(\quadnode_k^0))
		\end{equation}
		where $\mmu_h\in[\Pcal_1(\tri)]^2$ and $\v_h(\Xbar)$ is a continuous piecewise linear
		polynomial on $\tri$.
		
		Let us now introduce the linear interpolant $\v_I\in[\Pcal_1(\tri)]^2$ of
		$\v_h(\Xbar)$, so that with simple manipulations we can write
		\begin{equation}\label{eq:manipulation_l2}
			\begin{aligned}
				E_\tri(\mmu_h\cdot\v_h(\Xbar)) = 
				&\int_\tri \mmu_h\cdot\big(\v_h(\Xbar)-\v_I\big)\,\ds
				+ \int_\tri \mmu_h\cdot\v_I\,\ds 
				- \abs{\tri}\sum_{k=1}^{K_0}\quadweigth_k^0\mmu_h(\quadnode_k^0)
				\cdot\v_I(\quadnode_k^0)\\
				&\DB-\BD  \abs{\tri} \sum_{k=1}^{K_0}\quadweigth_k^0 \mmu_h(\quadnode_k^0)
				\cdot \big(\v_h(\Xbar(\quadnode_k^0))-\v_I(\quadnode_k^0)\big)
			\end{aligned}
		\end{equation}
		and we can study each term separately.
		
		The first term can be handled applying the Cauchy--Schwarz inequality,
		a classical interpolation result, and Lemma~\ref{le:duran}. Therefore, for $0\le s<1/2$,  we have
		\begin{equation}
			\label{eq:1st_term_l2}
			\aligned
			\abs{\int_\tri \mmu_h\cdot\big(\v_h(\Xbar)-\v_I\big)\,\ds}
			&\le \norm{\mmu_h}_{0,\tri}\norm{\v_h(\Xbar)-\v_I}_{0,\tri}\\
			&\le h_\tri^{1+s} \norm{\mmu_h}_{0,\tri} |\v_h(\Xbar)|_{1+s,\tri}\\
			&\le \frac{h_\tri^{1+s}}{1-2s}\norm{\mmu_h}_{0,\tri} |\v_h(\Xbar)|_{1,\tri}.
			\endaligned
		\end{equation}
		
		The second term is the quadrature error of the product $\mmu_h\cdot\v_I$ which
		can be estimated as
		\begin{equation}
			E_\tri(\mmu_h\cdot\v_I) = \int_\tri \mmu_h\cdot\v_I\,\ds
			- \abs{\tri}\sum_{k=1}^{K_0}\quadweigth_k^0\mmu_h(\quadnode_k^0)
			\cdot\v_I(\quadnode_k^0) = 0
		\end{equation}
		from the construction of $\v_I$ and the choice of quadrature rule.
		
		Now, in order to estimate the last term, we use Cauchy--Schwarz
		inequality and the exactness of the quadrature rule
		\begin{equation}\label{eq:3rd_part1_l2}
			\begin{aligned}
				&\abs{\abs{\tri} \sum_{k=1}^{K_0}\quadweigth_k^0 \mmu_h(\quadnode_k^0)\cdot
					\big(\v_h(\Xbar(\quadnode_k^0))-\v_I(\quadnode_k^0)\big)}\\
				&\qquad \le \bigg(\abs{\tri}
				\sum_{k=1}^{K_0}\quadweigth_k^0\,\mmu_h(\quadnode_k^0)^2 \bigg)^{1/2} \,
				\bigg(\abs{\tri}
				\sum_{k=1}^{K_0}\quadweigth_k^0\,
				\big(\v_h(\Xbar(\quadnode_k^0))-\v_I(\quadnode_k^0)\big)^2\bigg)^{1/2}\\
				&\qquad =\bigg(\int_\tri \abs{\mmu_h}^2\bigg)^{1/2}\, \bigg(\abs{\tri}
				\sum_{k=1}^{K_0}\quadweigth_k^0\,
				\big(\v_h(\Xbar(\quadnode_k^0))-\v_I(\quadnode_k^0)\big)^2\bigg)^{1/2}\\
				&\qquad \le K_0\abs{\tri}^{1/2}\norm{\mmu_h}_{0,\tri}
				\norm{\v_h(\Xbar)-\v_I}_{\infty,\tri}
			\end{aligned}
		\end{equation}
		The estimate for $\norm{\v_h(\Xbar)-\v_I}_{\infty,\tri}$ is not trivial.
		Since $\v_h(\Xbar)$ is piecewise linear in $\tri$, we have from
		Lemma~\ref{le:duran}, that
		$\v_h(\Xbar)\in \H^{1+s}(\tri)$ with $0\le s<1/2$.
		
		As mentioned before, using standard arguments for finite elements (see, \cite{ciarlet2002finite}),
		$\tri$ is affine equivalent to a reference element $\reftri$ and thanks to the
		inclusion $H^{1+s}(\reftri)\subset L^\infty(\reftri)$, we can
		write 
		\begin{equation}\label{eq:3rd_part3_l2}
			\begin{aligned}
				\norm{\v_h(\Xbar)-\v_I}_{\infty,\tri}
				&\le \norm{\widehat{\v_h(\Xbar)}-\widehat{\v_I}}_{\infty,\reftri}\\
				&\le \norm{\mathcal{I}-\widehat{\Pi}}_{\mathscr{L}(H^{1+s}(\reftri),L^\infty(\reftri))}\, \inf_{\widehat{\mathbf{q}}\in[\Pcal_1(\reftri)]^2} \norm{\widehat{\v_h(\Xbar)}+\widehat{\mathbf{q}}}_{1+s,\reftri}\\
				&\le C \abs{\widehat{\v_h(\Xbar)}}_{1+s,\reftri}
			\end{aligned}
		\end{equation}
		where $\widehat{\v_h(\Xbar)}$ is defined via the affine mapping $F_\tri$ as $\widehat{\v_h(\Xbar(\widehat{\s}))} = \v_h(\Xbar(F_\tri(\widehat{\s})))$, $\mathcal{I}$ is the identity operator and $\widehat{\Pi}$ the interpolation operator.
		
		Now, following \cite{dupont1980polynomial} and using the definition of fractional Sobolev seminorm, we have
		\small
		\begin{equation*}
			\begin{aligned}
				&\abs{\widehat{\v_h(\Xbar)}}_{1+s,\reftri}^2
				= \abs{\grads\widehat{\v_h(\Xbar)}}_{s,\reftri}^2
				= \int_{\reftri} \int_{\reftri} \frac{\abs{\grads\widehat{\v_h(\Xbar(\widehat{\s}_1))}-\grads\widehat{\v_h(\Xbar(\widehat{\s}_2))}}^2 }{\abs{\widehat{\s}_1-\widehat{\s}_2}^{2(s+1)}} \,d\widehat{\s}_1d\widehat{\s}_2\\
				&\quad= \abs{\det B_\tri}^{-2} \int_{\tri} \int_{\tri} \frac{\abs{B_\tri \big(\grads\v_h(\Xbar(\s_1))-\grads\v_h(\Xbar(\s_2))\big)}^2}{\abs{\s_1-\s_2}^{2(s+1)}} \bigg(\frac{\abs{\s_1-\s_2}}{\abs{B_\tri^{-1}(\s_1-\s_2)}}\bigg)^{2(s+1)}\,\ds_1\ds_2\\
				&\quad\le \frac{\norm{B_\tri}^{4+2s}}{\abs{\det B_\tri}^2} \abs{\grads\v_h(\Xbar)}_{s,\tri}^2 
			\end{aligned}
		\end{equation*}
		\normalsize
		so that
		\begin{equation}\label{eq:3rd_part4_l2}
			\abs{\widehat{\v_h(\Xbar)}}_{1+s,\reftri}^2 \le C\,h_\tri^{2s} \norm{\grads\v_h(\Xbar)}_{s,\tri}^2.
		\end{equation}
		
		By applying again Lemma~\ref{le:duran}, putting together
		\eqref{eq:3rd_part1_l2}, \eqref{eq:3rd_part3_l2}, and \eqref{eq:3rd_part4_l2} 
		and taking into account that $\abs{\tri}\le h_\tri^2$, we get
		\begin{equation}\label{eq:3rd_term_l2}
			\begin{aligned}
				\abs{\abs{\tri} \sum_{k=1}^{K_0}\quadweigth_k \mmu_h(\quadnode_k^0)\cdot \big(\v_h(\Xbar(\quadnode_k^0))-\v_I(\quadnode_k^0)\big)}
				&\le \frac{C}{1-2s} \abs{\tri}^{1/2} h_\tri^s
				\norm{\mmu_h}_{0,\tri} \DB\norm{\grads\v_h(\Xbar)}_{0,\tri}\BD\\
				&\le C\frac{h_\tri^{1+s}}{1-2s}\norm{\mmu_h}_{0,\tri}
				\norm{\grads\v_h(\Xbar)}_{0,\tri}.
			\end{aligned}
		\end{equation}
		
		Finally, for $0\le s<1/2$, the local estimate reads: 
		\begin{equation}\label{eq:local_l2_estimate}
			\abs{E_\tri(\mmu_h\cdot\v_h(\Xbar))}
			\le C \frac{h_\tri^{1+s}}{1-2s}
			\norm{\mmu_h}_{0,\tri}\DB\norm{\grads\v_h(\Xbar)}_{0,\tri}\BD.
		\end{equation}
		Taking $s=\frac12+\frac1{\log h_\tri}$ in the above inequality, we obtain with
		simple computations the following local estimate:
		\begin{equation}
			\abs{E_\tri(\mmu_h\cdot\v_h(\Xbar))}\le Ch_\tri^{3/2}|\log h_\tri|
			\norm{\mmu_h}_{0,\tri}\DB\norm{\grads\v_h(\Xbar)}_{0,\tri}\BD.
			\label{eq:disallineata}
		\end{equation}
		Summing on all the triangles in $\T_{h,2}^\B$, we get the global
		estimate~\eqref{eq:l2_quad_error}.
	\end{proof}
	
	We now study the following proposition for the estimate of the error for the $\LdBd$ scalar product of gradients.

	\begin{prop}\label{prop:l2grad_estimate}
		Let $\Vline_h$ and $\LL_h$ given by \eqref{eq:element}. Let us assume that $\Xbar$ satisfies Assumption~\ref{assum:xbar}.
		Given a quadrature rule $\{(\quadnode_k^1,\quadweigth_k^1)\}_{k=1}^{K_1}$ exact for constants, i.e. such that
		\begin{equation}
			\widehat{E}(\widehat{f}) = 0
			\quad\forall\widehat{f}\in\Pcal_0(\reftri),
		\end{equation}
		and a  quasi-uniform mesh  $\T_h^\Omega$, the following estimate holds
		\begin{equation}\label{eq:h1_quad_error}
			\abs{E_\B(\grads\mmu_h:\grads\v_h(\Xbar))}
			\le
			C\bigg( h_\B^{1/2}|\log h_\B^{min}| + \frac{h_\B}{h_\Omega} \bigg)
			\norm{\grads \mmu_h}_{0,\B}\norm{\Grad\v_h}_{0,\Omega}
		\end{equation}
		for all $\mmu_h\in\LL_h$ and $\v_h\in\Vline_h$. Here $E_\B$ is the sum of $E_\tri(\grads\mmu_h:\grads\v_h(\Xbar))$ for all $\tri\in\T_h^\B$ and  $h_\B^{min}=\min_{\tri\in\T_h^\B} h_\tri$. 
	\end{prop}
	
	\begin{proof}
		Using the same technique as in the proof of Proposition~\ref{prop:l2_quad_error},
		we partition the solid mesh $\T_h^\B$ as in
		\eqref{eq:solid_mesh_partition} and we work locally in an element
		$\tri\in\T_{h,2}^\B$: this means that we have to bound the difference
		\begin{equation}
			E_\tri(\grads\mmu_h:\grads\v_h(\Xbar)) = \int_\tri
			\grads\mmu_h:\grads\v_h(\Xbar)\,\ds -
			\abs{\tri}\sum_{k=1}^{K_1}\quadweigth_k^1\grads\mmu_h(\quadnode_k^1):
			\grads\v_h(\Xbar(\quadnode_k^1)).
		\end{equation}
		We recall that $\grads\mmu_h\in[\Pcal_0(\tri)]^2$, while $\grads\v_h(\Xbar)$ is
		a discontinuous piecewise constant function in~$\tri$.
		Working again with the interpolant $\v_I\in[\Pcal_1(\tri))]^2$ of $\v_h(\Xbar)$,
		we can write
		\begin{equation}\label{eq:manipulation_h1}
			\begin{aligned}
				E_\tri(\grads\mmu_h:\grads\v_h(\Xbar)) = &\int_\tri \grads\mmu_h:\big(\grads\v_h(\Xbar)-\grads\v_I\big)\,\ds\\
				&+ \int_\tri \grads\mmu_h:\grads\v_I\,\ds - \abs{\tri}\sum_{k=1}^{K_1}\quadweigth_k^1\grads\mmu_h(\quadnode_k^1):\grads\v_I(\quadnode_k^1)\\
				&\DB-\BD  \abs{\tri} \sum_{k=1}^{K_1}\quadweigth_k^1 \grads\mmu_h(\quadnode_k^1): \big(\grads\v_h(\Xbar(\quadnode_k^1))-\grads\v_I(\quadnode_k^1)\big)
			\end{aligned}
		\end{equation}
		so that we can deal separately with each term.
		
		For the first term, we use the Cauchy--Schwarz inequality
		\begin{equation}
			\abs{\int_\tri
				\grads\mmu_h:\big(\grads\v_h(\Xbar)-\grads\v_I\big)\,\ds}\le
			\norm{\grads\mmu_h}_{0,\tri} \norm{\grads\v_h(\Xbar)-\grads\v_I}_{0,\tri},
		\end{equation}
		then using Lemma~\ref{le:duran} we get
		\[
		\norm{\grads\v_h(\Xbar)-\grads\v_I}_{0,\tri}
		\le h_\tri^s\norm{\grads\v_h(\Xbar)}_{s,\tri}
		\le \frac{h_\tri^s}{1-2s}\norm{\grads\v_h(\Xbar)}_{0,\tri}
		\]
		so that, taking $s=\frac12+\frac1{\log h_\tri}$, we find
		\begin{equation}
			\begin{aligned}
				&\abs{\int_\tri \grads\mmu_h:\big(\grads\v_h(\Xbar)-\grads\v_I\big)\,\ds}\\
				&\hspace{3cm}\le Ch_\tri^{1/2}\abs{\log h_\tri} \norm{\grads\mmu_h}_{0,\tri}\norm{\grads\v_h(\Xbar)}_{0,\tri}.
			\end{aligned}
		\end{equation}
		
		The second term is zero by construction of $\v_I$ and the choice of quadrature rule
		\begin{equation}
			E_\tri(\grads\mmu_h:\grads\v_I) = \int_\tri \grads\mmu_h:\grads\v_I\,\ds - \abs{\tri}\sum_{k=1}^{K_1}\quadweigth_k^1\grads\mmu_h(\quadnode_k^1):\grads\v_I(\quadnode_k^1)=0.
		\end{equation}
		
		For the last term, we proceed as in \eqref{eq:3rd_part1_l2} so that we have
		\begin{equation}\label{eq:3rd_part1_h1}
			\begin{aligned}
				&\abs{\abs{\tri} \sum_{k=1}^{K_1}\quadweigth_k^1 \grads\mmu_h(\quadnode_k^1): \big(\grads\v_h(\Xbar(\quadnode_k^1))-\grads\v_I(\quadnode_k^1)\big)}\\
				&\hspace{3.5cm}\le K_1 \abs{\tri}^{1/2} \norm{\grads\mmu_h}_{0,\tri}\norm{\grads\v_h(\Xbar)-\grads\v_I}_{\infty,\tri}
			\end{aligned}
		\end{equation}
		
		In order to estimate $\norm{\grads\v_h(\Xbar)-\grads\v_I}_{\infty,\tri}$,
		we take into account the discontinuity of $\grads\v_h(\Xbar)$ in $\tri$:
		hence we consider again the tessellation of $\tri$ into disjoint polygons
		$P_j$, $j=1,\dots,J$, with $\tri=\bigcup_{j=1}^J P_j$
		and $\grads\v_h(\Xbar)$ constant in each $P_j$.
		Using that $\norm{\grads\v_I}_{\infty,\tri} \le\norm{\grads\v_h(\Xbar)}_{\infty,\tri} $ ,
		we can write
		\begin{equation}\label{eq:step_polygons}
			\begin{aligned}
				\norm{\grads\v_h(\Xbar)-\grads\v_I}_{\infty,\tri} 
				&= \max_{j=1,\dots,J} \norm{\grads\v_h(\Xbar)-\grads\v_I}_{\infty,P_j}\\
				&\le 2 \max_{j=1,\dots,J} \norm{\grads\v_h(\Xbar)}_{\infty,P_j}.
			\end{aligned}
		\end{equation}
		We notice that each polygon $P_j$ is the inverse image
		through $\Xbar$ of $\Xbar(\tri)\cap\trifj$ with $\trifj\in\T_{h/2}^\Omega$.
		Moreover, we denote by $\omega_\tri$ the macroelement
		consisting of all fluid triangles intersecting $\Xbar(\tri)$,
		i.e. ${\omega_\tri = \bigcup_{j=1}^J \trifj}$.
		Therefore, by applying the inverse inequality
		$\|\w_h\|_{\infty,\tau}\le C_I/h_\tau\|\w_h\|_{0,\tau}$ for $\w_h\in\Vline_h$
		and $\tau\in\T_{h/2}^\Omega$, we get
		\begin{equation}\label{eq:step_macroelement}
			\begin{aligned}
				\max_{j=1,\dots,J} \norm{\grads\v_h(\Xbar)}_{\infty,P_j}
				&\le \DB C_{\Xbar}\BD\max_{j=1,\dots,J} \norm{\Grad\v_h}_{\infty,\trifj}\\
				&\le \DB C_{\Xbar}\BD\max_{j=1,\dots,J}\,\frac{\DB C_{\trifj}\BD}{h_{\trifj}}
				\norm{\Grad\v_h}_{0,\trifj}
				\le \frac{\DB C\BD}{h_{\Omega}} \norm{\Grad\v_h}_{0,\omega_\tri},
			\end{aligned}
		\end{equation}
		\DB where $C_{\Xbar}$ depends on $\|\Xbar\|_{W^{1,\infty}(\B)}$ and\BD, in the last
		inequality, we exploited that $\T_{h/2}^\Omega$ is quasi-uniform. Putting
		together \eqref{eq:3rd_part1_h1} with \eqref{eq:step_polygons} and
		\eqref{eq:step_macroelement}, we end up with
		\begin{equation}
			\begin{aligned}
				&\abs{\abs{\tri} \sum_{k=1}^{K_1}\quadweigth_k^1 \grads\mmu_h(\quadnode_k^1): \big(\grads\v_h(\Xbar(\quadnode_k^1))-\grads\v_I(\quadnode_k^1)\big)}\\
				& \hspace{1.5cm} \le C \frac{\abs{\tri}^{1/2}}{h_{\Omega}} \norm{\grads\mmu_h}_{0,\tri} \norm{\Grad\v_h}_{0,\omega_\tri}\
				\le C\frac{h_\B}{h_{\Omega}} \norm{\grads\mmu_h}_{0,\tri} \norm{\Grad\v_h}_{0,\omega_\tri}.
			\end{aligned}
		\end{equation}
		Finally, summing on over all solid elements and exploiting $\Xbar(\B)\subset\Omega$, the global estimate \eqref{eq:h1_quad_error} is proven.
		
	\end{proof}
	
	
	The following results for the coupling term are direct consequence of the two propositions above.
	
	\begin{prop}\label{prop:l2_final}
		Under the same hypotheses of Proposition~\ref{prop:l2_quad_error}, if the continuous and discrete coupling bilinear forms are given by~\eqref{eq:ch_l2} and~\eqref{eq:noint_l2}, respectively, then the following quadrature error estimate holds
		\begin{equation}\label{eq:l2_est_fin}
			\abs{\c_0(\mmu_h,\v_h(\Xbar))-\c_{0,h}(\mmu_h,\v_h(\Xbar))} \le
			C h_\B^{3/2}|\log h_\B^{min}| \norm{\mmu_h}_{0,\B}\norm{\v_h}_{1,\Omega}
		\end{equation}
		for all $\mmu_h\in\LL_h$, $\v_h\in\Vline_h$.
	\end{prop}
	
	\begin{proof}
		The left hand side of~\eqref{eq:l2_est_fin} is exactly $E_\B(\mmu_h\cdot\v_h(\Xbar))$, therefore Proposition~\ref{prop:l2_quad_error} gives the result.
	\end{proof}
	Thanks to the above proposition, the final error estimate in the case
	when $\c=\c_0$ reads: 
	\begin{corollary}
		Within the setting of Theorem~\ref{theo:strang} and Proposition~\ref{prop:l2_final}, the following estimate holds true
		\begin{equation*}
			\begin{aligned}
				&\norm{\u-\u_h^\star}_{1,\Omega} + \norm{p-p_h^\star}_{0,\Omega} + \norm{\X-\X_h^\star}_{1,\B} + \norm{\llambda-\llambda_h^\star}_{\LL} \\
				&\le C\,\bigg(\inf_{\v_h\in\Vline_h}\norm{\u-\v_h}_{1,\Omega} + \inf_{q_h\in Q_h}\norm{p-q_h}_{0,\Omega} + \inf_{\Y_h\in\Sline_h}\norm{\X-\Y_h}_{1,\B} + \inf_{\mmu_h\in\LL_h}\norm{\llambda-\mmu_h}_{\LL} \\
				&\quad+h_\B^{1/2}|\log h_\B^{min}|\|\u_h\|_{1,\Omega}+\frac{ h_\B^{3/2}}{h_\B^{min}}|\log h_\B^{min}|\|\llambda_h\|_{\LL} + h_\Omega^s\,|\f|_{s,\Omega} + h_\B^s\,|\g|_{s,\B} +\, h_\B^{r}\|\d\|_{r,\B}\bigg).
			\end{aligned}
		\end{equation*}
		\label{co:L2}
	\end{corollary}
	\begin{proof}
		Proposition~\ref{prop:l2_final} yields that the Assumption~\ref{assump:quad_er} is satisfied with ${\rho_0(h)=h_\B^{3/2}|\log h_\B^{min}|}$. The inequality~\eqref{eq:l2_est_fin} is optimal for the regularity of the involved functions, but at the right hand side we need to have $\norm{\llambda_h}_{\LL}$, which is bounded thanks to the stability of the discrete problem (see Assumption~\ref{assump:infsup_appr} and the next section). By using the inverse inequality reported in Proposition~\ref{prop:inv_in}, we obtain
		\begin{equation}\label{eq:inv_in}
			\| \mmu_h \|_{0,\B}  \le  \frac{C_{-1}}{h_\B^{min}} \|\mmu_h\|_{-1,\B} \le  \frac{C_{-1}}{h_\B^{min}} \|\mmu_h\|_{\LL}\qquad\forall\mmu_h\in\LL_h,
		\end{equation}
		so that
		$$
		\rho_0(h)\|\llambda_h\|_{0,\B} \le C_{-1} \frac{ h_\B^{3/2}}{h_\B^{min}}|\log h_\B^{min}|\|\llambda_h\|_{\LL}.
		$$
	\end{proof}
	
	\begin{prop}\label{prop:h1_final}
		Under the same hypotheses of Propositions~\ref{prop:l2_quad_error} and \ref{prop:l2grad_estimate}, if the continuous and discrete coupling bilinear forms are given by~\eqref{eq:ch_h1} and~\eqref{eq:noint_h1}, respectively, then the following quadrature error estimate holds
		\begin{equation}\label{eq:h1_est_fin}
			\abs{\c_1(\mmu_h,\v_h(\Xbar))-\c_{1,h}(\mmu_h,\v_h(\Xbar))}\le
			C \bigg( \big( h_\B^{3/2} + h_\B^{1/2}\big)|\log h_\B^{min}| + \frac{h_\B}{h_\Omega} \bigg) \norm{\mmu_h}_{1,\B}\norm{\v_h}_{1,\Omega}
		\end{equation}
		for all $\mmu_h\in\LL_h$, $\v_h\in\Vline_h$.
	\end{prop}
	
	\begin{proof}
		The left hand side of~\eqref{eq:h1_est_fin} is the sum of $E_\B(\mmu_h\cdot\v_h(\Xbar))$ and ${E_\B(\grads\mmu_h:\grads\v_h(\Xbar))}$, therefore Propositions~\ref{prop:l2_quad_error} and~\ref{prop:l2grad_estimate} give the result.
	\end{proof}
	In the case of $\c=\c_1$, we have the following corollary of Theorem~\ref{theo:strang}.
	\begin{corollary}
		Within the setting of Theorem~\ref{theo:strang} and Proposition~\ref{prop:h1_final}, the following estimate holds true
		\begin{equation*}
			\begin{aligned}
				&\norm{\u-\u_h^\star}_{1,\Omega} + \norm{p-p_h^\star}_{0,\Omega} + \norm{\X-\X_h^\star}_{1,\B} + \norm{\llambda-\llambda_h^\star}_{1,\B} \\
				&\le C\,\bigg(\inf_{\v_h\in\Vline_h}\norm{\u-\v_h}_{1,\Omega} + \inf_{q_h\in Q_h}\norm{p-q_h}_{0,\Omega} + \inf_{\Y_h\in\Sline_h}\norm{\X-\Y_h}_{1,\B} + \inf_{\mmu_h\in\LL_h}\norm{\llambda-\mmu_h}_{1,\B} \\
				&\quad+\big( ( h_\B^{3/2} + h_\B^{1/2})|\log h_\B^{min}| + \frac{h_\B}{h_\Omega} \big)(\|\u_h\|_{1,\Omega}+\|\llambda_h\|_{1,\B})
				+ h_\Omega^s\,|\f|_{s,\Omega} + h_\B^s\,|\g|_{s,\B} +\,  h_\B^{r-1}\|\d\|_{r,\B}\bigg).
			\end{aligned}
		\end{equation*}
	\end{corollary}
	
	We conclude this section by proving, in a general framework, the inverse inequality we used in~\eqref{eq:inv_in}.
	\begin{prop}\label{prop:inv_in}
		Let $\T_h$ be a mesh of the domain $D\subset\R^2$ and let
		$V_h$ be a finite element space of degree $m$. Then, there exists a constant $C$ such that for all discrete function $v_h\in V_h$ it holds
		\[
		\sum_{\tri\in\T_h}h_\tri^2\|v_h\|_{0,\tri}^2 \le C\|v_h\|_{-1,D}^2,
		\]
		where $C$ depends on the degree of $V_h$ and on the shape regularity constant of the mesh~$\T_h$.
	\end{prop}
	\begin{proof}
		It follows from a result by Schatz and Wahlbin \cite[Lemma 1.1]{SW} that 
		\[
		\sum_{\tri\in\T_h}\|v\|_{-1,\tri}^2\le \|v\|_{-1,D}^2, \qquad \forall v\in H^{-1}(D).
		\]
		If we prove that
		\begin{equation}\label{ineq: inverse1}
			\|v_h\|_{0,\tri}\le Ch_\tri^{-1}\|v_h\|_{-1,\tri}\qquad \forall v_h\in \Pcal_m(\tri),
		\end{equation}
		with $C$ depending only on $m$ and the aspect ratio of $\tri$, then the statement follows by
		\[
		\sum_{\tri\in \T_h}h_\tri^2\|v_h\|_{0,\tri}^2\le C\sum_{\tri\in \T_h}\|v_h\|_{-1,\tri}^2\le C\|v_h\|_{-1,D}^2.
		\]
		It remains to prove \eqref{ineq: inverse1}, which follows by standard scaling arguments. Let $\reftri$ be a reference triangle, and $F_\tri$ a linear function mapping $\reftri$ onto $\tri$, and define $\widehat v_h$ by $\widehat v_h(\widehat{\x})=v_h({\x})$ with ${\x}=F_\tri(\widehat{\x})$. Then
		\begin{equation}\label{ineq: 2norm}
			\|v_h\|_{0,\tri}\le Ch_\tri\|\widehat v_h\|_{0,\reftri},
		\end{equation}
		and by equivalence of norms in $\Pcal_m(\reftri)$ we have
		\begin{equation}\label{ineq: equiv}
			\|\widehat v_h\|_{0,\reftri}\le C\|\widehat v_h\|_{-1,\reftri}.
		\end{equation}
		By definition
		\[
		\|\widehat v_h\|_{-1,\reftri}=\sup_{\widehat w\in H_0^1(\reftri)}\frac{\left(\widehat v_h,\widehat w\right)_{\reftri}}{\|\widehat w\|_{1,\reftri}}.
		\]
		But, since $\widehat w(\widehat{\x})=w({\x})$,
		\[
		\left(\widehat v_h,\widehat w\right)_{\reftri}\le Ch_\tri^{-2}\left(v_h,w\right),
		\]
		and
		\[
		\|\widehat w\|_{1,\reftri} \ge C h_\tri^{-1}\left(\sum_{i=0}^1 h_\tri^{2i}|w|^2_{i,\tri}\right)^{1/2}\ge C \|w\|_{1,\tri}.
		\]
		Therefore
		\begin{equation}\label{ineq: -tnorm}
			\|\widehat v_h\|_{-1,\reftri}\le Ch_\tri^{-2}\|v_h\|_{-1,\tri}.
		\end{equation}
		Then, from \eqref{ineq: 2norm}--\eqref{ineq: -tnorm} we
		obtain~\eqref{ineq: inverse1}.
	\end{proof}
	
	\DB
	
	\begin{rem}
		When dealing with fluid-structure interaction problems, the solution of the
		fictitious domain approach is singular across the interface between fluid and
		solid.
		Typically, for instance, $\u$ belongs to
		$\mathbf{H}^{3/2-\varepsilon}(\Omega)$ but not in $\mathbf{H}^{3/2}(\Omega)$
		because of the jump of the normal stresses. It is then natural to consider
		\emph{optimal} an error estimate like the one of Corollary~\ref{co:L2} which
		gives a rate of convergence of $O(h^{1/2})$ up to the logarithmic factor.
		
		It could be interesting however to see if, in case of smoother solutions, a
		higher rate of convergence could be obtained which fully exploits the
		polynomial order used for the approximation. It turns out that this is the
		case as we are going to show now.
		
	\end{rem}
	
	The following proposition refines the estimate of Corollary~\ref{co:L2} in the
	case of smooth solutions.
	
	\begin{prop}
		Let us assume that $(\u,p,\X,\llambda)$ belongs to $\mathbf{H}^2(\Omega)\times
		H^1(\Omega)\times\mathbf{H}^2(\B)\times\mathbf{L}^2(\B)$ and that
		\begin{equation}\label{eq:regu+}
			\|\u\|_{2,\Omega}+\|p\|_{1,\Omega}+\|\X\|_{2,\B}+\|\llambda\|_{0,\B}\le
			C(\|\f\|_{0,\Omega}+\|\g\|_{0,\B}+\|\d\|_{2,\B}).
		\end{equation}
		Moreover, let us assume for simplicity that the right hand sides in
		Problem~\ref{pro:approx_stationary_general} are computed exactly and that the
		mesh $\T_h^\B$ is quasiuniform. Then under the same hypotheses as in
		Corollary~\ref{co:L2} we have
		\[
		\begin{aligned}
			\norm{\u-\u_h^\star}_{1,\Omega}&+\norm{p-p_h^\star}_{0,\Omega}+
			\norm{\X-\X_h^\star}_{1,\B}+\norm{\llambda-\llambda_h^\star}_{\LL}\\
			&\le C\max(h_\Omega,h_\B)\big(|\u|_{2,\Omega}+|p|_{1,\Omega}+
			|\X|_{2,\B}+\|\llambda\|_{0,\B}\big).
		\end{aligned}
		\]
	\end{prop}
	
	\begin{proof}
		
		As for the proof of Corollary~\ref{co:L2}, we need to estimate the various
		consistency terms appearing in Equation~\eqref{eq:strang_4}. Here we comment
		on how to improve the bound of the following two terms under our regularity
		assumptions:
		\begin{equation}\label{eq:sups}
			\begin{aligned}
				&\sup_{\mmu_h\in\LL_h}\frac{|\c_0(\mmu_h,\u_h(\Xbar))-\c_{0,h}(\mmu_h,\u_h(\Xbar))|}{\|\mmu_h\|_{\LL}},\\
				&\sup_{\v_h\in\Vline_h}\frac{|\c_0(\llambda_h,\v_h(\Xbar))-\c_{0,h}(\llambda_h,\v_h(\Xbar))|} {\|\v_h\|_{1,\Omega}},
			\end{aligned}
		\end{equation}
		here $\llambda_h$ and $\u_h$ are components of the solution to
		Problem~\ref{pro:discrte_stationary_general}.
		
		We estimate the first term in~\eqref{eq:sups}. 
		We observe that~\eqref{eq:disallineata} holds true even when the function
		$\v_h(\Xbar)$ is replaced by $\v_h(\Xbar)-\w_h$ where $\w_h$ is any function
		of $\S_h$.
		Hence, we introduce the piecewise linear interpolant $\II_\B\u(\Xbar)\in\S_h$
		of $\u(\Xbar)$ on the mesh $\T_\B$ and proceed as follows.
		\[
		\aligned
		&\c_0(\mmu_h,\u_h(\Xbar))-\c_{0,h}(\mmu_h,\u_h(\Xbar))\\
		&\quad=\c_0(\mmu_h,\u_h(\Xbar)-\II_\B\u(\Xbar))
		-\c_{0,h}(\mmu_h,\u_h(\Xbar)-\II_\B\u(\Xbar))\\
		&\quad\le C\sum_{\tri\in\T_\B}h^{3/2}_\tri{|\log h_\tri|}
		\|\mmu_h\|_{0,\tri}\|{\grads(\u_h(\Xbar)-\II_\B\u(\Xbar))}\|_{0,\tri}\\
		&\quad\le C\sum_{\tri\in\T_\B}h^{3/2}_\tri{|\log h_\tri|}
		\|\mmu_h\|_{0,\tri}(\|{\grads(\u_h(\Xbar)-\u(\Xbar))}\|_{0,\tri}\\
		&{\hspace{5cm}}+\|{\grads(\u(\Xbar)-\II_\B\u(\Xbar))}\|_{0,\tri})\\
		&\quad\le Ch^{3/2}_\B{|\log h_\B|}\|\mmu_h\|_{0,\B}\|\u_h-\u\|_{1,\Omega}
		+C\sum_{\tri\in\T_\B}h^{3/2}_\tri{|\log h_\tri|}\|\mmu_h\|_{0,\tri}h_\tri|\u(\Xbar)|_{2,\tri}\\
		&\quad\le Ch^{1/2}_\B{|\log h_\B|}\|\mmu_h\|_\LL\|\u_h-\u\|_{1,\Omega}
		+Ch^{3/2}_\B{|\log h_\B|}\|\mmu_h\|_\LL|\u|_{2,\Omega},
		\endaligned
		\]
		where in the last inequality we used the inverse estimate
		in~\eqref{eq:inv_in}.
		
		The regularity assumption~\eqref{eq:regu+} and Theorem~\ref{theo:brezzi} imply
		\begin{equation}
			\label{eq:XX_h}
			\begin{aligned}
				\norm{\u-\u_h}_{1,\Omega}& + \norm{\X-\X_h}_{1,\B}\\
				&\le
				C\,\big(h_\Omega(\|\u\|_{2,\Omega}+\|p\|_{1,\Omega})+
				h_\B(\|\X\|_{2,\B}+\|\llambda\|_{0,\B})\big).
			\end{aligned}
		\end{equation}
		We thus have
		\begin{equation}\label{eq:bound_u}
			\begin{aligned}
				&\sup_{\mmu_h\in\LL_h}\frac{|\c_0(\mmu_h,\u_h(\Xbar))-\c_{0,h}(\mmu_h,\u_h(\Xbar))|}{\|\mmu_h\|_{\LL}}\\
				&\quad\le C\,h_\B^{1/2}|\log h_\B| \max(h_\Omega,h_\B)\bigg(\|\u\|_{2,\Omega}+\|p\|_{1,\Omega}+\|\X\|_{2,\B}+\|\llambda\|_{0,\B}\bigg)\\
				&\qquad\qquad+ C\,h_\B(\|\u\|_{2,\Omega}+\|\X\|_{2,\B}).\\
			\end{aligned}
		\end{equation}
		
		The second term in~\eqref{eq:sups} can be bounded using Proposition~\ref{prop:l2_final} as
		follows:
		\[
		\sup_{\v_h\in\Vline_h}
		\frac{|\c_0(\llambda_h,\v_h(\Xbar))-\c_{0,h}(\llambda_h,\v_h(\Xbar))|}
		{\|\v_h\|_{1,\Omega}} \le
		C h_\B^{3/2}|\log h_\B| \|\llambda_h\|_{0,\B}.
		\]
		We now estimate $\|\llambda_h\|_{0,\B}$ exploiting the regularity of
		$\llambda$. We consider the $L^2$ projection $\Pi^0:\LdBd\rightarrow\LL_h$, that is for any ${\boldsymbol{\eta}\in\LdBd}$, $(\boldsymbol{\eta}-\Pi^0\boldsymbol{\eta},\mmu_h)=0$ for all $\mmu_h\in\LL_h$. The projection has the following properties thanks to the fact that $\LL_h=\Sline_h$:
		\begin{equation}\label{eq:proj_prop}
			\|\Pi^0\boldsymbol{\eta}\|_{0,\B}\le\|\boldsymbol{\eta}\|_{0,\B},
			\qquad\qquad
			\|\boldsymbol{\eta}-\Pi^0\boldsymbol{\eta}\|_\LL \le C\,h_\B\|\boldsymbol{\eta}\|_{0,\B}.
		\end{equation}
		In order to estimate $\|\llambda_h\|_{0,\B}$, we bound
		$\|\llambda_h-\llambdaI\|_{0,\B}$ taking into account that $\S_h=\LL_h$
		\begin{equation}
			\aligned
			\|\llambdaI-\llambda_h\|_{0,\B}&
			= \frac{(\llambdaI-\llambda_h,\llambdaI-\llambda_h)_\B}{\|\llambdaI-\llambda_h\|_{0,\B}}
			\le
			\sup_{\Y_h\in\S_h}\frac{\c_0(\llambdaI-\llambda_h,\Y_h)}{\|\Y_h\|_{0,\B}}\\
			&
			=\sup_{\Y_h\in\S_h}\frac{\c_0(\llambdaI-\llambda,\Y_h)+\c_0(\llambda-\llambda_h,\Y_h)}
			{\|\Y_h\|_{0,\B}}.
			\endaligned
		\end{equation}
		Since $\llambdaI$ is the $L^2$ projection of $\llambda$ in $\LL_h=\S_h$, the
		term $\c_0(\llambdaI-\llambda_h,\Y_h)$ vanishes. 
		Moreover, the difference of the solid equations in
		Problems~\ref{pro:stationary_general} and~\ref{pro:discrte_stationary_general}
		gives
		\[
		\c_0(\llambda-\llambda_h,\Y_h)=\a_s(\X-\X_h,\Y_h)
		\quad\forall\Y_h\in\Sline_h.
		\]
		Putting together the last two relations, a standard inverse inequality yields
		\begin{equation}
			\begin{aligned}
				\|\llambdaI-\llambda_h\|_{0,\B}
				&\le \sup_{\Y_h\in\S_h}\frac{\a_s(\X-\X_h,\Y_h)}{\|\Y_h\|_{0,\B}}\\
				&\le C\,\sup_{\Y_h\in\S_h} \frac{\|\X-\X_h\|_{1,\B}\|\Y_h\|_{1,\B}}{\|\Y_h\|_{0,\B}}\\
				&\le C\, h_\B^{-1}\|\X-\X_h\|_{1,\B}.
			\end{aligned}
		\end{equation}
		The bound~\eqref{eq:XX_h} implies that
		\begin{equation}
			\begin{aligned}
				\|\llambda_h\|_{0,\B}&\le\|\llambdaI\|_{0,\B}+\|\llambdaI-\llambda_h\|_{0,\B}\\
				&\le \|\llambda\|_{0,\B}+
				C\,\big(h_\B^{-1}h_\Omega(\|\u\|_{2,\Omega}+\|p\|_{1,\Omega})
				+\|\X\|_{2,\B}+\|\llambda\|_{0,\B}\big)
			\end{aligned}
		\end{equation}
		so that $\llambda_h$ is bounded in $\mathbf{L}^2(\B)$.
		Hence, we have the following estimate for the second term in~\eqref{eq:sups}
		\begin{equation}
			\label{eq:bound_l}
			\aligned
			\sup_{\v_h\in\Vline_h}&
			\frac{|\c_0(\llambda_h,\v_h(\Xbar))-\c_{0,h}(\llambda_h,\v_h(\Xbar))|}{\|\v_h\|_{1,\Omega}}
			\le C h_\B^{3/2}|\log h_\B| \|\llambda_h\|_{0,\B}\\
			&\le  C h_\B^{3/2}|\log h_\B|\|\llambda\|_{0,\B}\\
			&\quad
			+Ch_\B^{3/2}|\log h_\B|\Big(h_\B^{-1}h_\Omega(\|\u\|_{2,\Omega}+\|p\|_{1,\Omega})
			+\|\X\|_{2,\B}+\|\llambda\|_{0,\B}\Big).
			\endaligned
		\end{equation}

		Finally, by combining the estimates in~\eqref{eq:bound_u}
		and~\eqref{eq:bound_l} with~\eqref{eq:strang_4} and
		Proposition~\ref{prop:l2_final}, we obtain the desired estimate.
		
	\end{proof}
	
	\BD
	
	\section{Stability proof with inexact integration}\label{sec:infsup}
	
	This section is devoted to check Assumption~\ref{assump:infsup_appr}, which
	corresponds to the discrete inf-sup condition in the case of inexact integration of the coupling term. We adopt the approach already used in \cite{stat} and based on \cite{xu2003some}. Exploiting the saddle point structure of $\A_h$ in operator form
	\begin{equation*}
		\left[\begin{array}{@{}cc:c@{}}
			\Amatr_f & 0 & \Cmatr_{f,h}^\top\\
			0 & \Amatr_s & -\Cmatr_s^\top \\
			\hdashline
			\Cmatr_{f,h} & -\Cmatr_s & 0
		\end{array}\right],
	\end{equation*}
	the proof splits into two steps: first, we prove the inf-sup condition for
	\begin{equation*}
		\C_h = [\Cmatr_{f,h},\,-\Cmatr_s]
	\end{equation*}
	and then we show that
	\begin{equation*}
		\left[\begin{array}{@{}cc@{}}
			\Amatr_f & 0\\
			0 & \Amatr_s
		\end{array}\right]
	\end{equation*}
	is elliptic in the kernel of $\C_h$.
	
	\begin{prop}\label{prop:infsupch}
		Let us assume that the $\LdBd$ term of $\c_h$ is computed with a quadrature rule $\{(\quadnode_k^0,\quadweigth_k^0)\}_{k=1}^{K_0}$ which is exact for quadratic polynomials, while the $\LdBd$ scalar product of gradients is approximated with a quadrature rule $\{(\quadnode_k^1,\quadweigth_k^1)\}_{k=1}^{K_1}$ which is exact for constants. Moreover, we assume that the $L^2$ projection $\projLdue$ from $\Hub$ to $\Sline_h$ is $H^1$ stable, that is $\|\projLdue\Y\|_{1,\B}\le C_0 \|\Y\|_{1,\B}$ for all $\Y\in\Hub$. Then, there exists a constant $\beta_0$ such that the following condition holds true
		\begin{equation}
			\sup_{(\v_h,\Y_h)\in\Vdiv \times \Sline_h} \frac{\c_h(\mmu_h,\v_h(\Xbar)-\Y_h)}{(\norm{\v_h}^2_{1,\Omega} + \norm{\Y_h}^2_{1,\B})^{1/2}} \ge \beta_0 \norm{\mmu_h}_{\LL}\qquad\forall\mmu_h\in\LL_h,
		\end{equation}
		where 
		\begin{equation*}
			\Vdiv = \{\v_h\in\Vline_h:(\div\v_h,q_h)_\Omega=0\quad\forall q_h\in Q_h\}.
		\end{equation*}
	\end{prop}
	
	\begin{proof}
		\textit{Case 1.} Let us consider $\c=\c_0$ and $\c_h=\c_{0,h}$ given by~\eqref{eq:ch_l2}. Using the definition of norm in the dual space $\LL=\Hubd$, there exists $\Ytilde\in\Hub$ realizing the supremum in the continuous case, so that
		\begin{equation}
			\norm{\mmu_h}_{\LL}
			= \sup_{\Y\in\Hub} \frac{(\mmu_h,\Y)_\B}{\norm{\Y}_{1,\B}}
			= \sup_{\Y\in\Hub} \frac{\c_0(\mmu_h,\Y)}{\norm{\Y}_{1,\B}}
			= \frac{\c_0(\mmu_h,\Ytilde)}{\|\Ytilde\|_{1,\B}}.
		\end{equation}
		By exploiting the projection $\projLdue$, we have
		\begin{equation}
			\frac{\c_0(\mmu_h,\Ytilde)}{\|\Ytilde\|_{1,\B}}
			= \frac{\c_0(\mmu_h,\projLdue\Ytilde)}{\|\Ytilde\|_{1,\B}}
			\le C_0 \frac{\c_0(\mmu_h,\projLdue\Ytilde)}{\|\projLdue\Ytilde\|_{1,\B}}
			\le \sup_{\Y_h\in\Sline_h} \frac{\c_0(\mmu_h,\Y_h)}{\norm{\Y_h}_{1,\B}}.
		\end{equation}
		Thanks to the exactness of the chosen quadrature rule, we can notice that $\c$ and $\c_h$ coincide in the solid domain, therefore
		\begin{equation}\label{eq:dim_infsup_L2}
			\begin{aligned}
				\norm{\mmu_h}_{\LL} &\le \sup_{\Y_h\in\Sline_h} \frac{\c_0(\mmu_h,\Y_h)}{\norm{\Y_h}_{1,\B}} =  \sup_{\Y_h\in\Sline_h} \frac{\c_{0,h}(\mmu_h,\Y_h)}{\norm{\Y_h}_{1,\B}}\\
				&\le \sup_{(\v_h,\Y_h)\in\Vdiv \times \Sline_h}
				\frac{\c_{0,h}(\mmu_h,\DB\Y_h-\v_h(\Xbar)\BD)}{(\norm{\v_h}^2_{1,\Omega} + \norm{\Y_h}^2_{1,\B})^{1/2}}.
			\end{aligned}
		\end{equation}
		\textit{Case 2.} Now, let us take $\c=\c_1$ and $\c_h=\c_{1,h}$ as defined in \eqref{eq:c_h1} and \eqref{eq:ch_h1}, respectively. Given $\mmu_h\in\LL_h$, we can take $\Y_h=\mmu_h$ so that
		\begin{equation}
			\begin{aligned}
				\norm{ \mmu_h }_{\LL} &= \frac{(\mmu_h,\Y_h)_\B +(\grads\mmu_h,\grads\Y_h)_\B}{\norm{\Y_h}_{1,\B}} = \frac{\c_1(\mmu_h,\Y_h)}{\norm{\Y_h}_{1,\B}}
				\le \sup_{\Y_h\in\Sline_h} \frac{\c_1(\mmu_h,\Y_h)}{\norm{\Y_h}_{1,\B}}.
			\end{aligned}
		\end{equation}
		Therefore, working as in \eqref{eq:dim_infsup_L2}, the result is proved.
	\end{proof}
	
	We now prove the ellipticity of $\A_h$ in the kernel of $\C_h$.
	
	\begin{prop} \label{prop:ellipticity} 
		Let us assume that the $\LdBd$ term of $\c_h$ is computed with a quadrature rule $\{(\quadnode_k^0,\quadweigth_k^0)\}_{k=1}^{K_0}$ which is exact for quadratic polynomials, while the $\LdBd$ scalar product of gradients is approximated with a quadrature rule $\{(\quadnode_k^1,\quadweigth_k^1)\}_{k=1}^{K_1}$ which is exact for constants.
		Then, there exists $\theta^\star>0$ independent on the mesh sizes such that
		\begin{equation}
			\a_f(\u_h,\u_h) + \a_s(\X_h,\X_h) \ge \theta^\star (\norm{\u_h}_{1,\Omega}^2 + \norm{\X_h}_{1,\B}^2)
		\end{equation}
		for all pairs $(\u_h,\X_h)$ in the kernel of $\C_h$ defined as
		\begin{equation}
			\K_{\C_h} = \{ (\v_h,\Y_h)\in\Vdiv\times\S_h:\c_h(\mmu_h,\v_h(\Xbar))-\c(\mmu_h,\Y_h)=0\quad\forall\mmu_h\in\LL_h\}.
		\end{equation}
	\end{prop}
	
	\begin{proof}
		Recalling the definition of the bilinear forms~\eqref{eq:forme}, if $\beta>0$, we simply have that
		\begin{equation}
			\begin{aligned}
				\a_f(\u_h,\u_h) + \a_s(\X_h,\X_h) &\ge C\norm{\u_h}_{1,\Omega}^2 + \beta \norm{\X_h}_{0,\B}^2 + \kappa \norm{\grads\X_h}_{0,\B}^2\\
				&\ge C\norm{\u_h}_{1,\Omega}^2 + \min\{\beta,\kappa\}\norm{\X_h}_{1,\B}^2.
			\end{aligned}
		\end{equation}
		and the property is satisfied. Otherwise, if $\beta=0$, we have that
		\begin{equation}\label{eq:beta_zero}
			\a_f(\u_h,\u_h) + \a_s(\X_h,\X_h) \ge C\norm{\u_h}_{1,\Omega}^2 + \kappa \norm{\grads\X_h}_{0,\B}^2;
		\end{equation}
		hence, in this situation, we have to control the missing term
		$\norm{\X_h}_{0,\B}$. We follow the sketch used in \cite{2021}. In this way,
		we can prove this proposition at once for both choices of $\c_h$.
		
		We introduce the mean of $\X_h$ as
		\begin{equation}\label{eq:mean}
			\Xcirc = \abs{\B}^{-1} \int_\B \X_h\,\ds 
		\end{equation}
		so that, using the Poincar\'{e} inequality, we can write
		\begin{equation}
			\|\X_h-\Xcirc\|_{0,\B} \leq C \norm{\grads\X_h}_{0,\B}
		\end{equation}
		and therefore
		\begin{equation} \label{eq:initial_bound}
			\norm{\X_h}_{0,\B} \le \|\Xcirc\|_{0,\B} + \|\X_h-\Xcirc\|_{0,\B} \le \|\Xcirc\|_{0,\B}+C\norm{\grads\X_h}_{0,\B}.
		\end{equation}
		Exploiting that $(\u_h,\X_h)\in\K_{\C_h}$, we have that
		\begin{equation}\label{eq:c_manip}
			\c(\mmu_h,\Xcirc) = \c_h(\mmu_h,\u_h(\Xbar)) - \c(\mmu_h,\X_h-\Xcirc)\quad\forall\mmu_h\in\LL_h
		\end{equation}
		hence, taking $\mmu_h=\Xcirc$, which is constant, we have 
		\begin{equation}
			\|\Xcirc\|_{0,\B}^2 = \c(\Xcirc,\Xcirc) = \c_{0,h}(\Xcirc,\u_h(\Xbar)) - \c(\Xcirc,\X_h-\Xcirc);
		\end{equation}
		therefore, using the definition of $\Xcirc$ in \eqref{eq:mean}, with some computations we can see that $\c(\Xcirc,\X_h-\Xcirc)=0$, indeed
		\begin{equation}\label{eq:computations}
			\c(\Xcirc,\X_h-\Xcirc) = \Xcirc \bigg( \int_\B \X_h\,\ds - \abs{\B}\Xcirc \bigg) = 0
		\end{equation}
		
		At this point, we need to find a bound for $\c_h(\Xcirc,\u_h(\Xbar))$. Adding and subtracting the same quantity, we can write
		\begin{equation}
			\c_{0,h}(\Xcirc,\u_h(\Xbar)) = \c_{0,h}(\Xcirc,\u_h(\Xbar)) - \c(\Xcirc,\u_h(\Xbar)) + \c(\Xcirc,\u_h(\Xbar));
		\end{equation}
		in particular, for continuity we have
		\begin{equation}
			\c(\Xcirc,\u_h(\Xbar)) \leq \|\Xcirc\|_{0,\B}\|\u_h(\Xbar)\|_{0,\B}
		\end{equation}
		while, applying the result of Proposition~\ref{prop:l2_final}, 
		\begin{equation}
			\c_{0,h}(\Xcirc,\u_h(\Xbar)) - \c(\Xcirc,\u_h(\Xbar)) \leq  C h_\B^{3/2}|\log h_\B^{min}|  \|\Xcirc\|_{0,\B}\|\u_h\|_{1,\Omega}
		\end{equation}
		so that, in combination with \eqref{eq:computations}, we get
		\begin{equation}\label{eq:final_bound_Xcirc}
			\|\Xcirc\|_{0,\B} \leq  C(1+h_\B^{3/2}|\log h_\B^{min}|) \norm{\u_h}_{1,\Omega}.
		\end{equation}
		
		Consequently, putting together \eqref{eq:final_bound_Xcirc} and \eqref{eq:initial_bound}, we have
		
		\begin{equation}
			\norm{\X_h}_{0,\B} \leq  C(1+h_\B^{3/2}|\log h_\B^{min}|) \norm{\u_h}_{1,\Omega} + C\norm{\grads\X_h}_{0,\B}
		\end{equation}
		and therefore, we can conclude the proof by saying that we can find a constant $\theta^\star$ such that
		\begin{equation}\label{eq:elker_a_ch}
			\a_f(\u_h,\u_h) + \a_s(\X_h,\X_h) \ge \theta^\star (\norm{\u_h}_{1,\Omega}^2 + \norm{\X_h}_{1,\B}^2).
		\end{equation}
	\end{proof}
	
	Putting together the results of the previous propositions, we find that the inf-sup condition for~$\A_h$ holds true.
	\begin{prop}
		There exists a positive constant $\theta^\star$ independent of h such that the following inf-sup condition holds true
		\begin{equation}
			\inf_{\U\in\K_{\BB_h}}\sup_{\V\in\K_{\BB_h}} \frac{\A_h(\U,\V)}{\normiii{\U}\normiii{\V}} \ge \theta^\star.
		\end{equation}
	\end{prop}
	
	\section{Numerical tests}\label{sec:numerical_tests}
	
	In this last section we present \DB three \BD numerical tests with the aim of assessing the quadrature error estimates discussed in Section~\ref{sec:quad_error}. \DB We then present a numerical test carried out by using quadratic finite elements\BD. The rate of the quadrature error is analyzed by measuring the difference between exact and approximate interface matrix and by comparing the convergence history of the numerical solutions. We consider the following particular case of Problem~\ref{pro:stationary_general}, with $\Omega=[-2,2]^2$ and $\B=[0,1]^2$.
	
	\begin{pro}
		\label{pro:stationary_test}
		Let $\Xbar\in\Winftyd$ be invertible with Lipschitz inverse. Given $\f\in \LdOd$, $\g\in \LdBd$ and $\d\in\LdBd$, find $(\u,p)\in\Huo \times \Ldo$, $\X\in \Hub$ and $\llambda\in \LL$, such that
		\begin{subequations}
			\begin{align}
				&(\Grad\u,\Grad\v)_\Omega-(\div\v,p)_\Omega+\c(\llambda,\v(\Xbar))=(\f,\v)_\Omega && 
				\forall\v\in\Huo\\
				&(\div\u,q)_\Omega=0&&\forall q\in \Ldo\\
				& (\grads\X,\grads\Y)_\B-\c(\llambda,\Y)=(\g,\Y)_\B&&\forall\Y\in \Hub \\
				& \c(\mmu,\u(\Xbar)-\X)=\c(\mmu,\d)&&\forall\mmu\in \LL
			\end{align}
		\end{subequations}
	\end{pro}
	
	We set the initial mapping $\Xbar$ to be
	\begin{equation*}
		\Xbar(\s) = (-0.62+2 s_1,-0.62+2 s_2),\qquad\s=(s_1,s_2),
	\end{equation*}
	so that the actual configuration of the immersed structure coincides with the square $\Omega_s=[-0.62,1.38]^2$. 
	
	For the finite element discretization of the above problem, we partition the domain $\Omega$ with a right-oriented uniform triangulation $\T_h^\Omega$, whereas a left-oriented uniform triangulation $\T_h^\B$ is chosen for the solid reference domain $\B$. An example of this configuration with coarse meshes is depicted in Figure~\ref{fig:test_configuration}. 
	
	\begin{figure}
		\subfloat[]{\includegraphics[trim=60 16 40 50,width=0.3\linewidth]{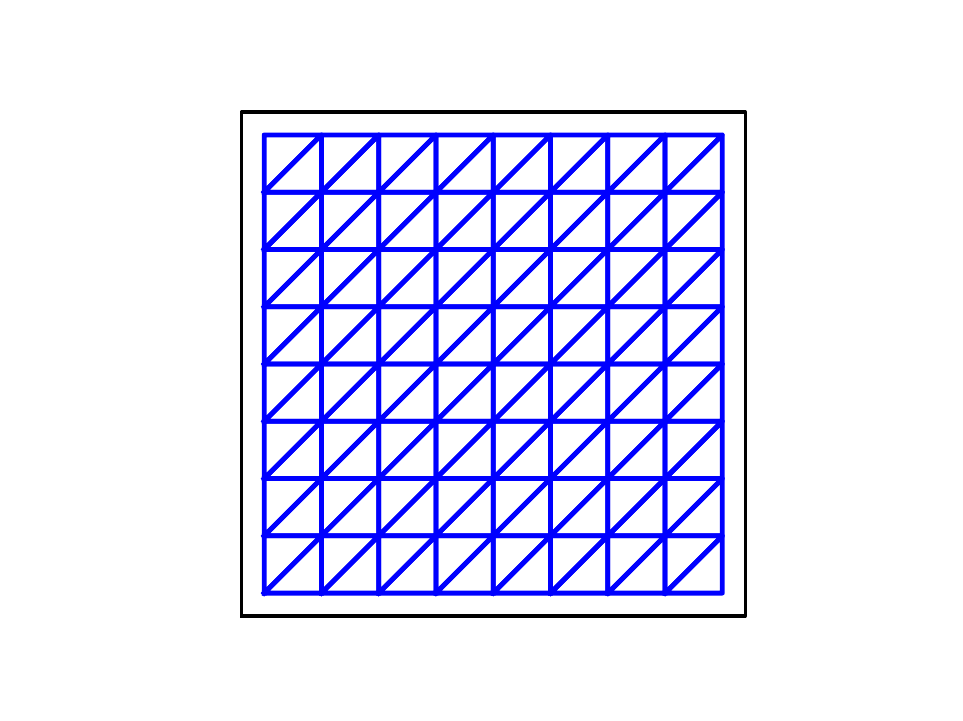}}
		\subfloat[]{\includegraphics[trim=60 16 40 50,width=0.3\linewidth]{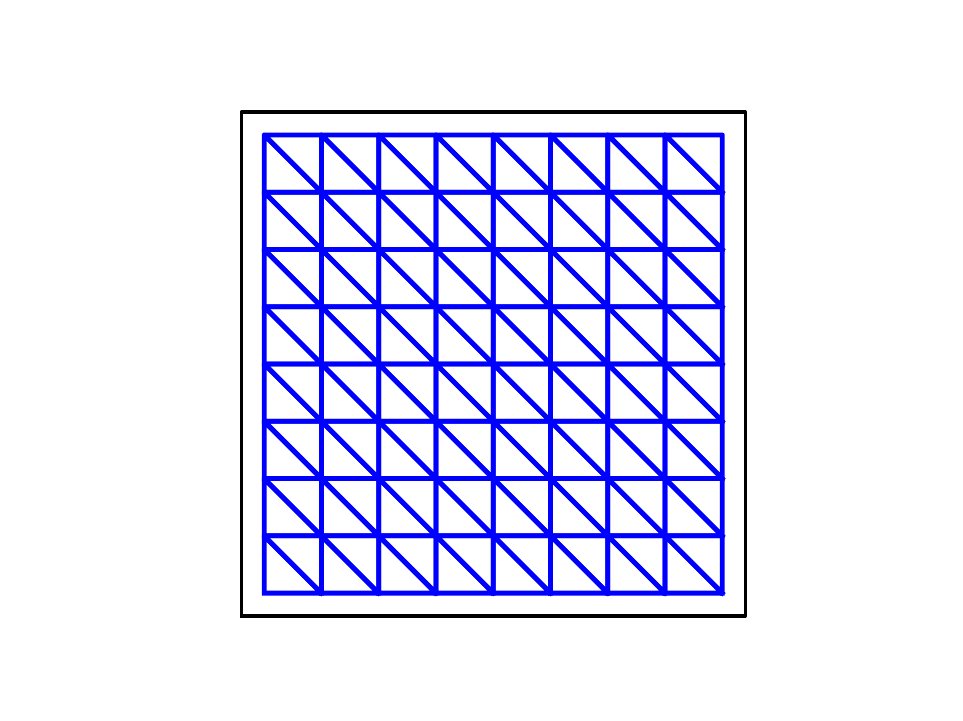}}
		\subfloat[]{\includegraphics[trim=60 16 40 50,width=0.3\linewidth]{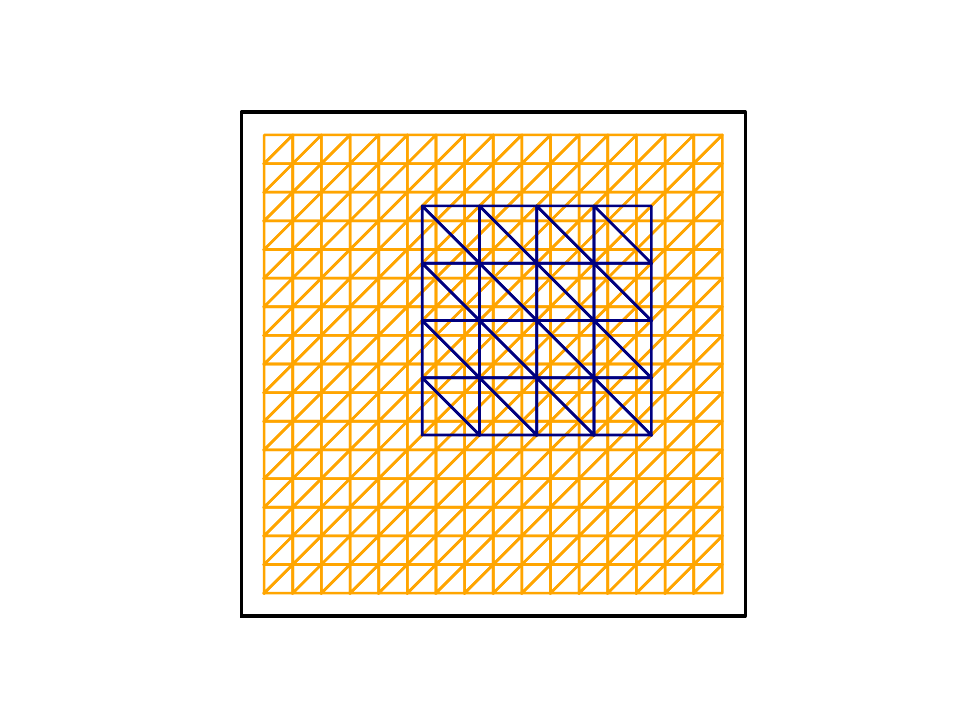}}
		\caption{Example of meshes used for the discretization of Problem~\ref{pro:stationary_test}. From left to right: a right-oriented uniform mesh (fluid), a left-oriented uniform mesh (solid) and the geometric configuration of the problem (fluid mesh in orange, mapped solid mesh in dark blue).}
		\label{fig:test_configuration}
	\end{figure}
	
	The coupling term is integrated (both in the exact and approximate case) with a precise enough quadrature rule. Moreover, in the case of $\c=\c_0$, the convergence of the Lagrange multiplier $\llambda$ is studied by considering the norm of the dual space $\LL_0=\Hubd$ we used at continuous level: this dual norm is computed by solving the associated equation ${-\Delta\Psi+\Psi=\llambda}$ with homogeneous Neumann boundary conditions.
	
	\begin{figure}
		\subfloat[\label{fig:quad_error_decay1}]{\includegraphics[width=0.45\textwidth]{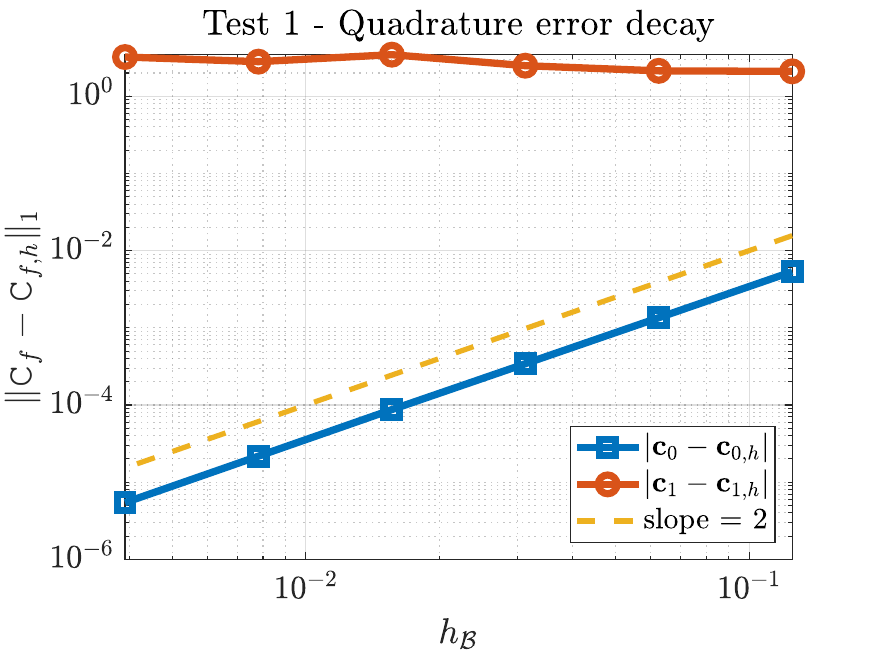}}\qquad
		\subfloat[\label{fig:quad_error_decay2}]{\includegraphics[width=0.45\textwidth]{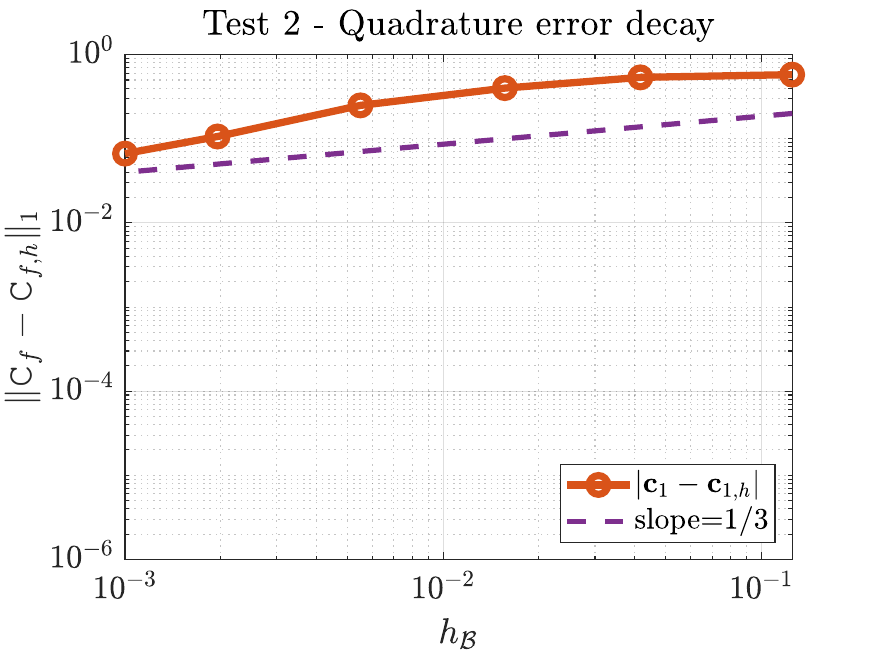}}
		\caption{Decay of the quadrature error committed when the coupling term is assembled with the approximate procedure. The quadrature error is measured by computing the 1--norm of the difference $\Cmatr_f - \Cmatr_{f,h}$. The results agree with the theoretical estimates presented in Section~\ref{sec:quad_error}. More precisely, the $\LdBd$ coupling (blue line) \lg converges with rate 2 \gl in Test~1 (${h_\B\rightarrow0}$, $h_\B/h_\Omega=1/2$). On the other hand, the error related to the $\Hub$ coupling (orange line) decays only in Test~2 ($h_\B\rightarrow0$, $h_\B/h_\Omega=\frac12h_\B^{1/3}\rightarrow0$) with rate $1/3$, as expected.} 
		\label{fig:quad_error_decay}
	\end{figure}
	
	\subsubsection*{Test 1}
	
	\DB
	We solve Problem~\ref{pro:stationary_test} by choosing the the right hand side $\f,\,\g,\,\d$ in such a way that we obtain an approximation of the following exact solution
	\begin{equation*}
		\begin{aligned}
			&\u(x,y) = \curl\big((4-x^2)^2(4-y^2)^2\big) &&\text{ in }\Omega\\
			&p(x,y) = 150\sin(x)&&\text{ in }\Omega\\
			&\X(s_1,s_2) = \curl\big((4-s_1^2)^2(4-s_2^2)^2\big)&&\text{ in }\B\\
			&\llambda(s_1,s_2) = \left( e^{s_1}, e^{s_2}\right)&&\text{ in }\B.
		\end{aligned}
	\end{equation*}
	The fluid variables are  discretized by the Bercovier--Pironneau element, while conforming piecewise linear elements are considered for the solid variables. As a consequence of using linear elements for approximating the velocity, we have that this is a low-order Stokes pair: indeed, both variables have optimal convergence rate equal to one. However, it has been observed numerically that the pressure usually superconverges, with rate equal to 3/2 \cite{mass}. The coupling term is assembled by considering a quadrature rule exact for quadratic polynomials.
	\BD
	
	For the initial configuration of this test, we choose a $16\times 16$ triangulation for the pressure mesh and a $8\times8$ mesh for $\B$. Then, we refine both meshes five times in such a way that $h_\B\rightarrow0$, while the ratio $h_\B/h_\Omega$ is kept constant. In particular, we have $h_\B/h_\Omega=1/2$.
	
	In Figure~\ref{fig:quad_error_decay1}, we plot the behavior of the
	quadrature error we commit on the coupling term by measuring the difference
	between $\Cmatr_f - \Cmatr_{f,h}$ with the 1--norm for matrices. We can see
	that the behavior of the error reflects perfectly our theoretical estimates:
	\lg the $\LdBd$ coupling shows a second order rate\gl, whereas the $\Hub$ coupling shows poor performance since $h_\B/h_\Omega$ is kept constant.
	
	As expected, the convergence of the solution is also influenced by the assembly technique chosen for the interface matrix. Convergence plots are collected in  Figure~\ref{fig:test1_results}: results for $\c=\c_0$ are reported on the left column, while the right column is related to $\c=\c_1$. The $\LdBd$ coupling term produces optimal results also when assembled in approximate way. On the other hand, only the exact computation of the interface matrix $\Cmatr_f$ produces an optimal method when the $\Hub$ scalar product is considered.
	
	\subsubsection*{Test 2}
	
	\DB Within the same framework of Test~1, we choose a different mesh refinement strategy. \BD In this case, for the initial configuration we have a $8\times8$ mesh for both pressure and solid domain $\B$. We then refine both meshes six times in such a way that $h_\B\rightarrow0$, but also $h_\B/h_\Omega\rightarrow0$. In particular, we choose $h_\B$ such that $h_\B=(\frac12h_\Omega)^{3/2}$, which implies $h_\B/h_\Omega=\frac12h_\B^{1/3}$.
	
	The behavior of $\Cmatr_f - \Cmatr_{f,h}$, depicted in
	Figure~\ref{fig:quad_error_decay2}, agrees with the theoretical estimates: the
	error originated from $\c_{1,h}$ converges with rate $1/3$ with respect to $h_\B$, which is consistent with our choice of meshes.
	
	The convergence history of the solution is reported in Figure~\ref{fig:test2_results}. Let us notice that the method with approximate $\Hub$ coupling, despite the presence of some oscillations, has overall behavior consistent with the exact case and theoretical results. Let us also point out that the suboptimal rate of convergence for the solid variables is $1/3$, which is caused by our choice of meshes.
	
	\begin{figure}
		\centering
		\textbf{Test 1: $h_\B\rightarrow0$\\$\LdBd$ coupling \textit{vs} $\Hub$ coupling }\\
		\vspace{3mm}
		\includegraphics[width=0.4\linewidth]{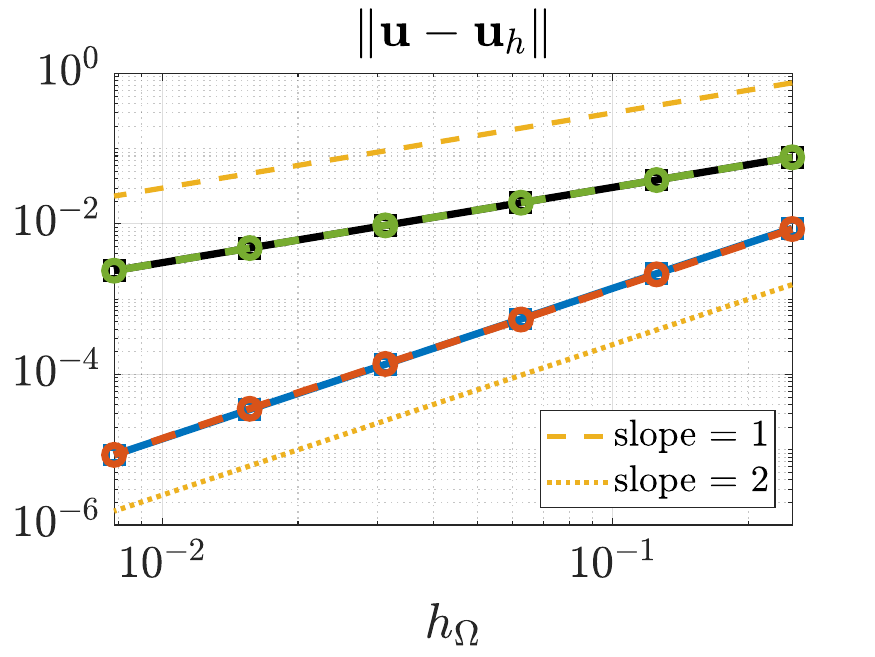}\qquad
		\includegraphics[width=0.4\linewidth]{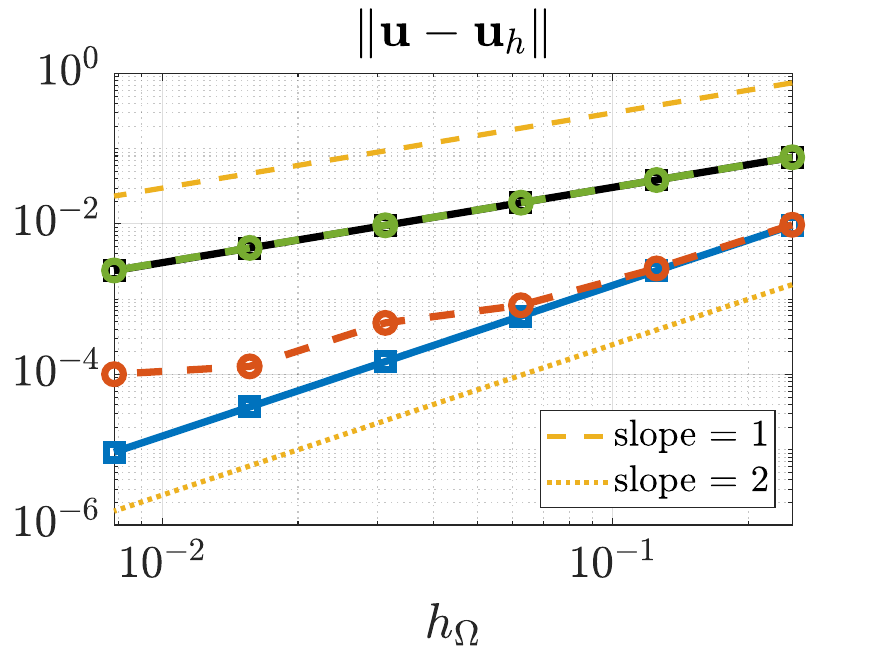}\\\medskip
		\includegraphics[width=0.4\linewidth]{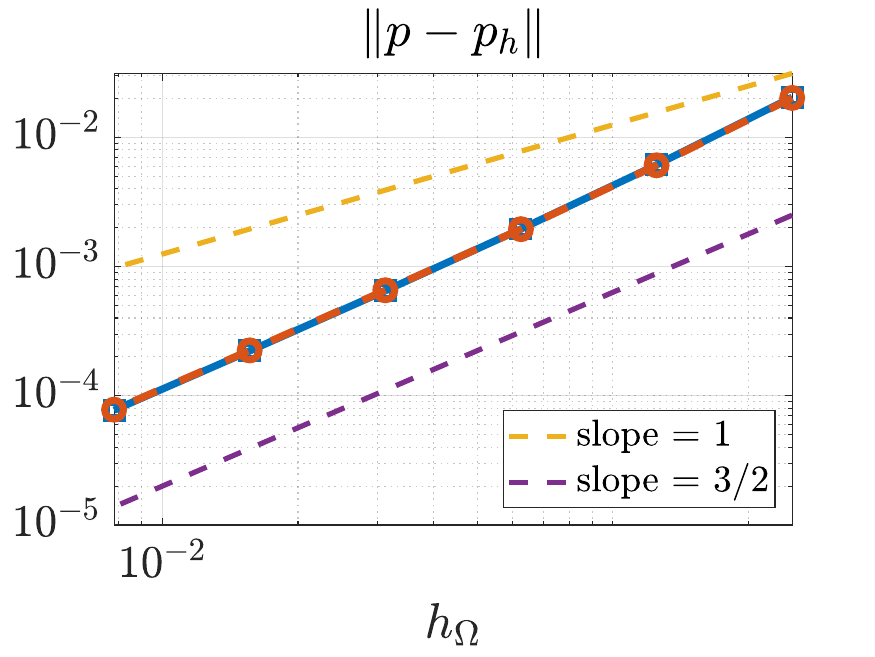}\qquad
		\includegraphics[width=0.4\linewidth]{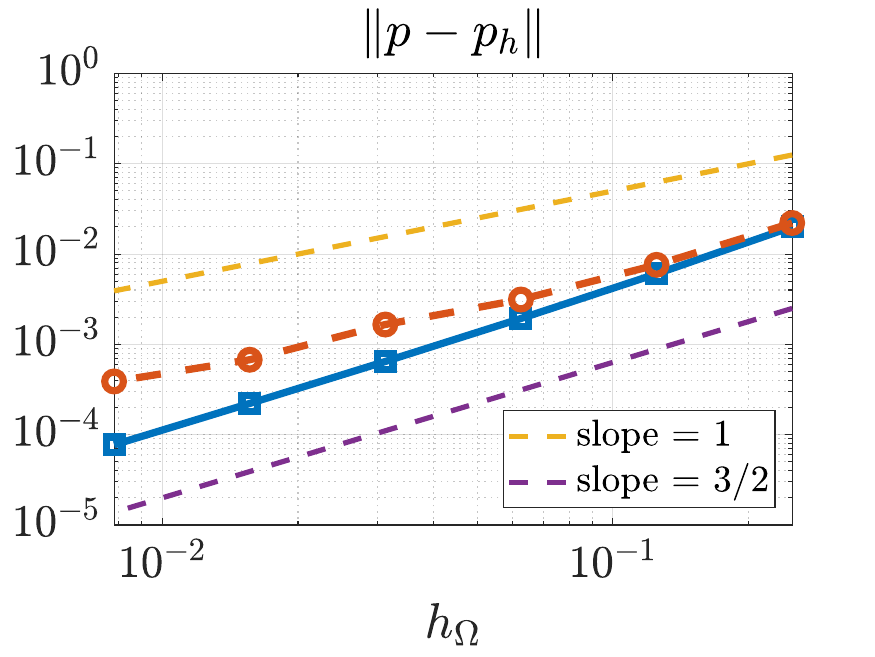}\\\medskip
		\includegraphics[width=0.4\linewidth]{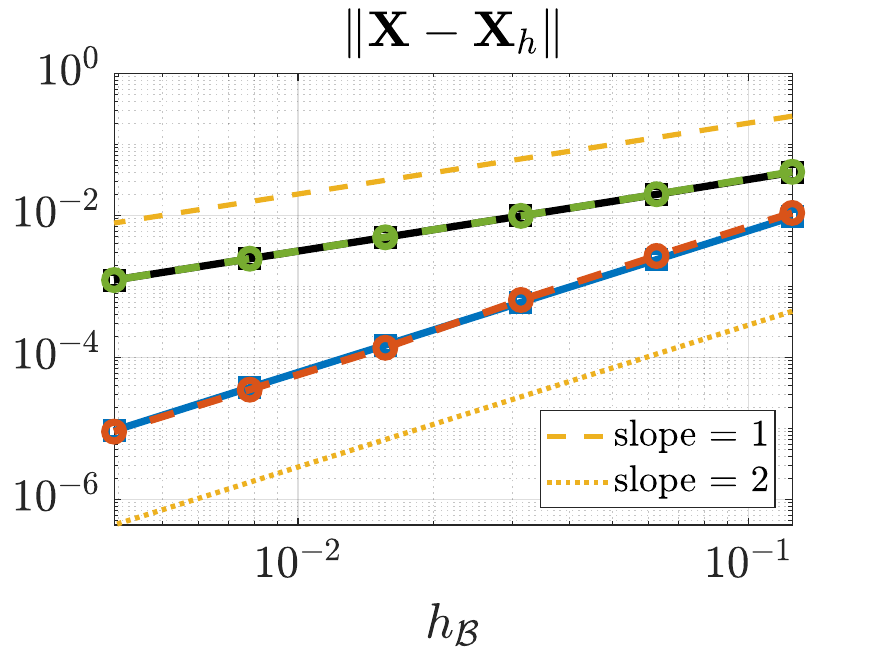}\qquad
		\includegraphics[width=0.4\linewidth]{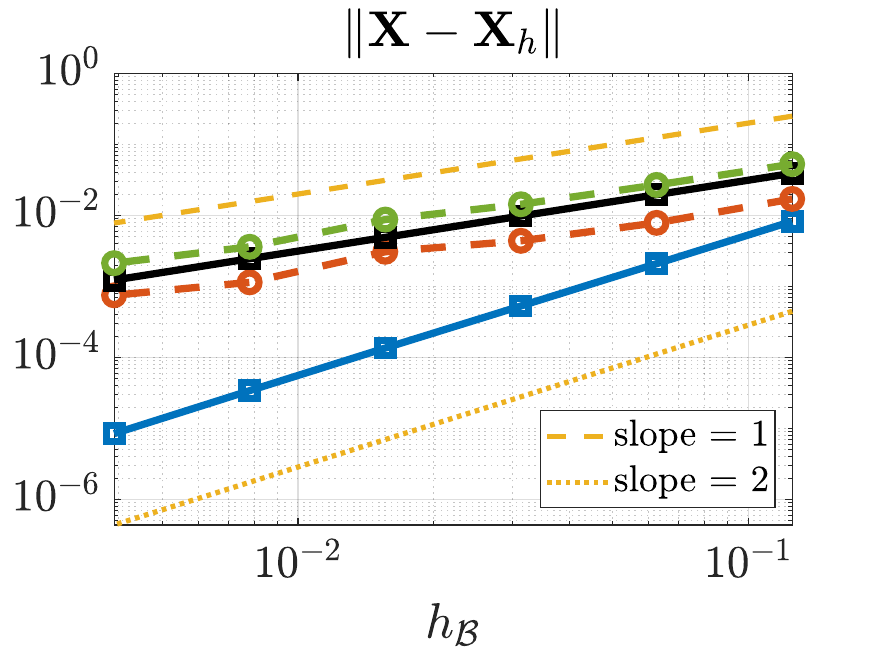}\\\medskip
		\includegraphics[width=0.4\linewidth]{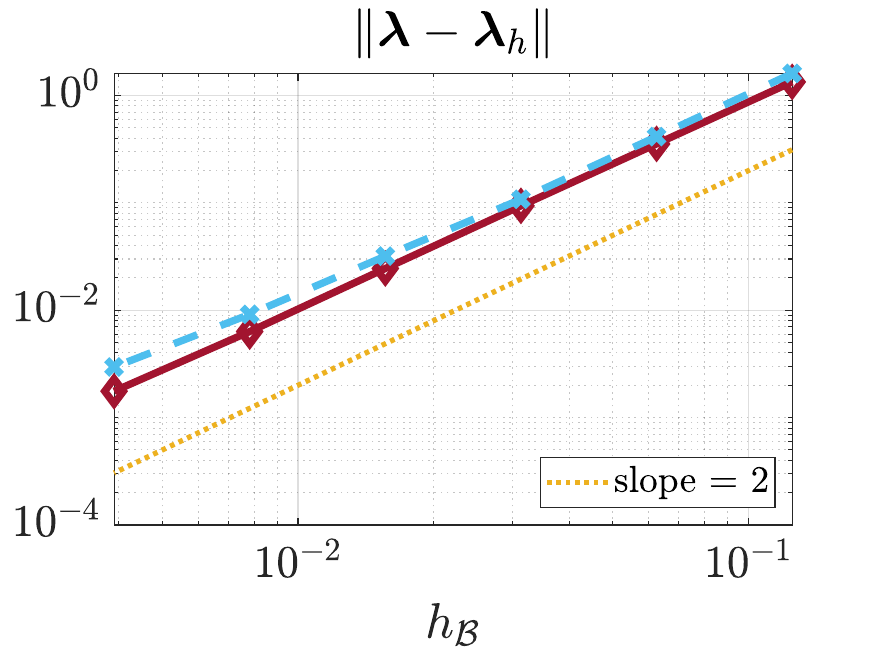}\qquad
		\includegraphics[width=0.4\linewidth]{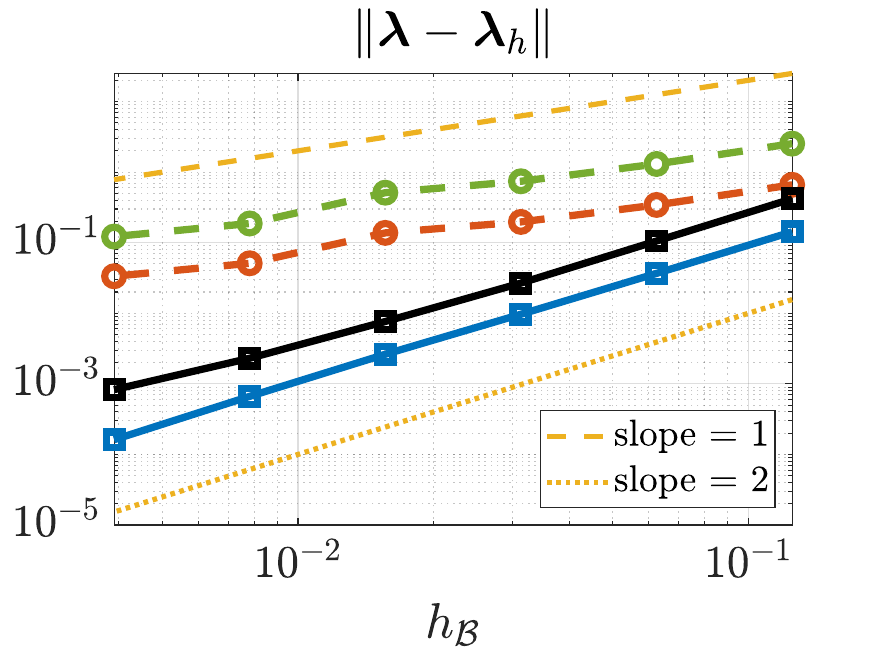}\\\medskip
		\includegraphics[width=1\linewidth]{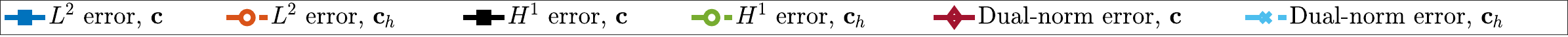}\\\medskip
		\caption{Convergence history of $\u,p,\X,\llambda$ in Test~1: comparison between exact and approximate assembly of the interface matrix. The results obtained with $\c=\c_0$ are collected in the left column, while the right column is related to $\c=\c_1$.}
		\label{fig:test1_results}
	\end{figure}
	
	\begin{figure}
		\centering
		\textbf{Test 2: $h_\B\rightarrow0$, $h_\B/h_\Omega\rightarrow0$\\$\Hub$ coupling }\\
		\vspace{3mm}
		
		\includegraphics[width=0.4\linewidth]{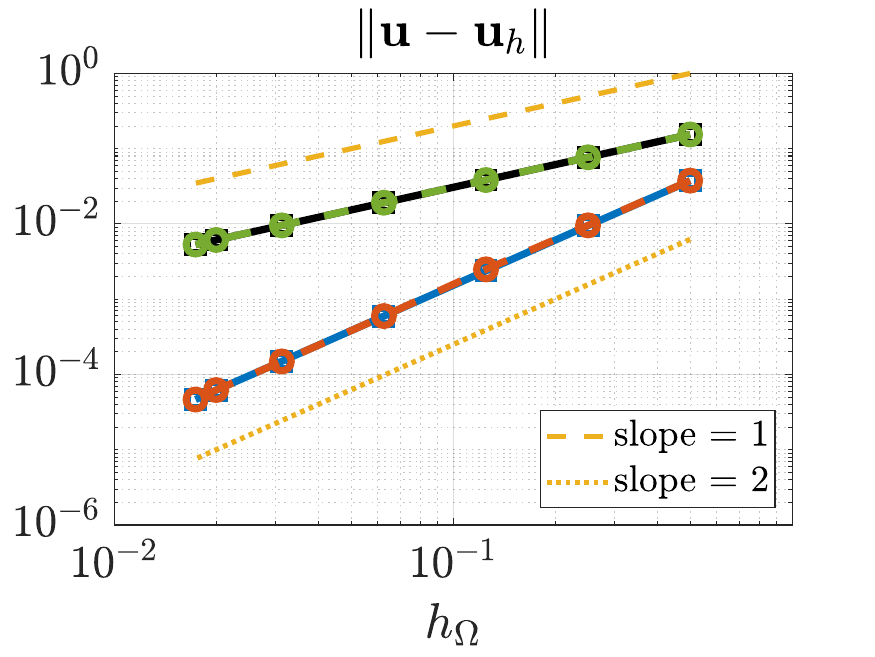}
		\includegraphics[width=0.4\linewidth]{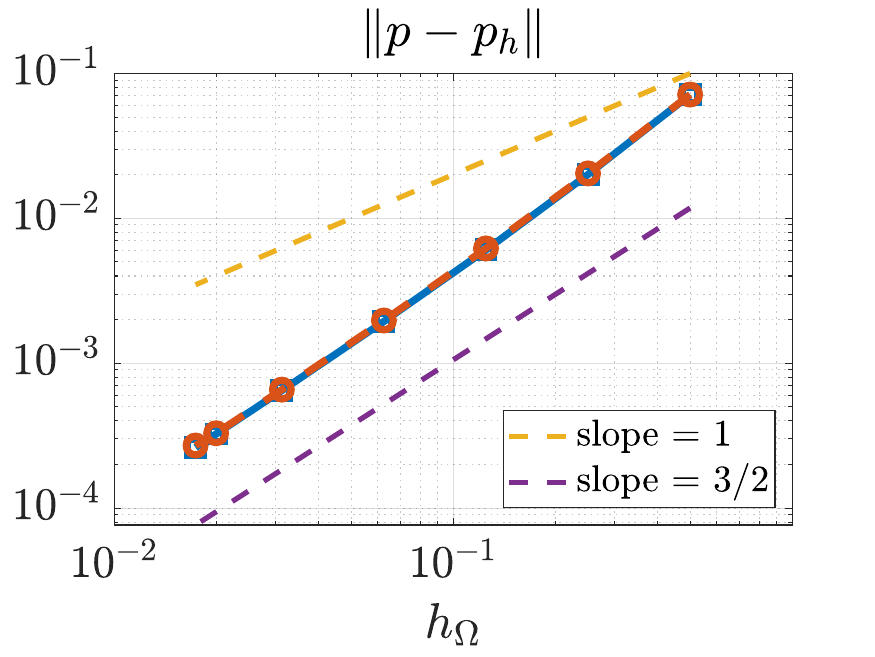}\\\medskip
		\includegraphics[width=0.4\linewidth]{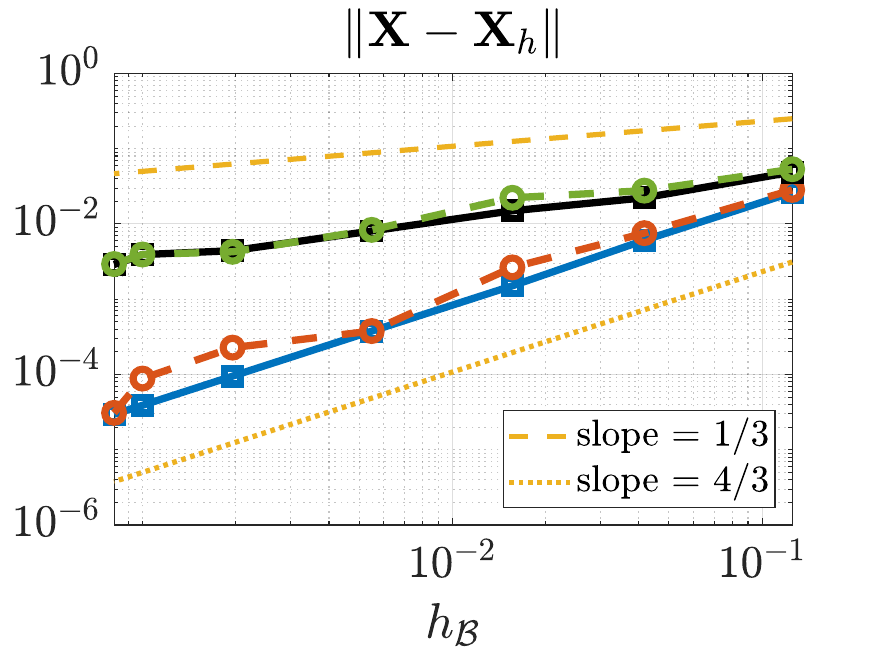}
		\includegraphics[width=0.4\linewidth]{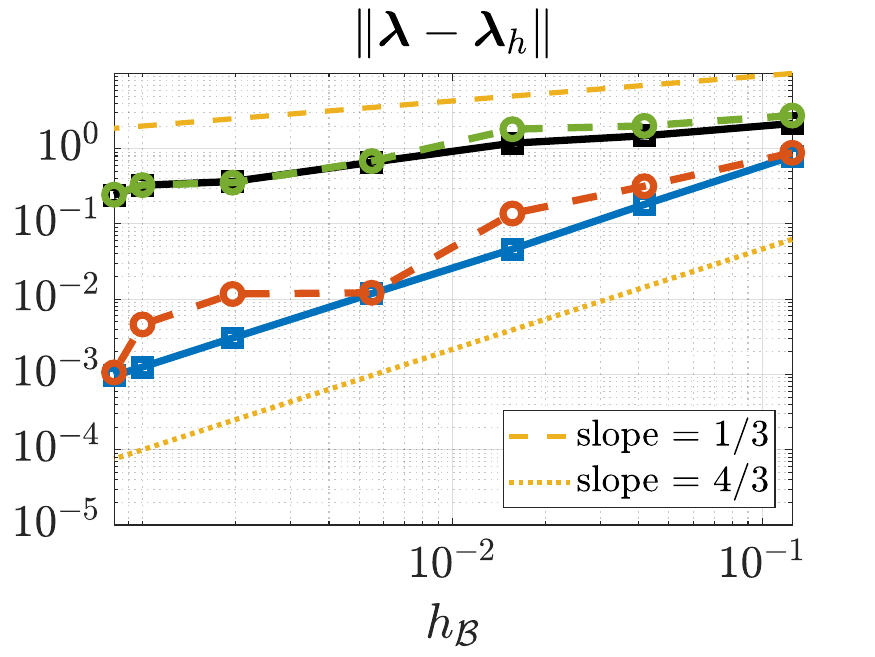}\\\medskip
		\includegraphics[width=0.6\linewidth]{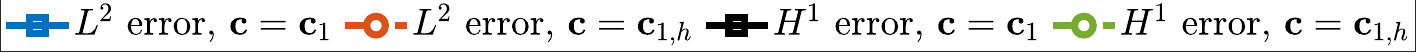}\\\medskip
		\caption{Convergence history of $\u,p,\X,\llambda$ in Test~2: comparison between exact and approximate assembly of the interface matrix for $\c=\c_1$.}
		\label{fig:test2_results}
	\end{figure}
	
	\DB
	\subsubsection*{Test 3}
	Within the same framework of Test~1, we choose the right hand side so that we obtain an approximation of the following exact solution
	\begin{equation*}
		\begin{aligned}
			&\u(x,y) = \begin{cases*}
				\u_f \text{ in }\Of\\
				\u_s \text{ in }\Omega\setminus\Of
			\end{cases*}\\
			&p(x,y) = \sin(x)&&\text{ in }\Omega\\
			&\X(s_1,s_2) = \big(-s_1\sin(s_1s_2),\,s_2\sin(s_1s_2)\big)&&\text{ in }\B\\
			&\llambda(s_1,s_2) = \left( e^{s_1},\,e^{s_2}\right)&&\text{ in }\B,
		\end{aligned}
	\end{equation*}
	where 
	\begin{equation*}
		\begin{aligned}
			&\u_f(x,y) = 10^{-3}\big((x^4 - 2x^3 + x^2)(4y^3 - 6y^2 + 2y),\,-(y^4 - 2y^3 + y^2)(4x^3 - 6x^2 + 2x)\big)&&\\
			&\u_s(x,y) = \u_f(x,y) + \curl\boldsymbol{\Phi}&&\\
			&\boldsymbol{\Phi}(x,y) = 50 (x+0.62)^2 (x-1.38)^2 (y+0.62)^2 (y-1.38)^2.
		\end{aligned}
	\end{equation*}
	We point out that the gradient of the velocity $\u$ jumps along the interface $\partial\Os$, therefore the theoretical rate of convergence for $\u$ is $1/2$ in $H^1$ norm. Convergence plots are reported in Figure~\ref{fig:test3_results}, where the left column is referred to $\c=\c_0$, while the right column to $\c=\c_1$. The $\LdBd$ coupling term provides optimal results when assembled exactly, while it is affected by the quadrature error when computed by means of the inexact procedure. The $\Hub$ coupling term provides optimal results only when computed exactly.
	
	\begin{figure}
		\centering
		\DB
		\textbf{Test 3: $h_\B\rightarrow0$\\$\LdBd$ coupling \textit{vs} $\Hub$ coupling }\\
		\vspace{3mm}\BD
		
		\includegraphics[width=0.4\linewidth]{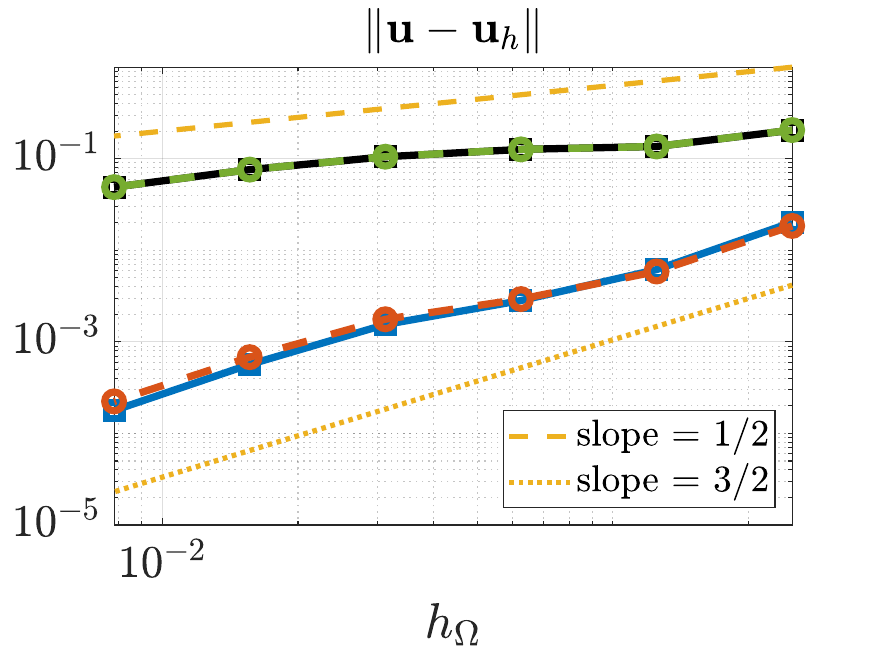}\qquad
		\includegraphics[width=0.4\linewidth]{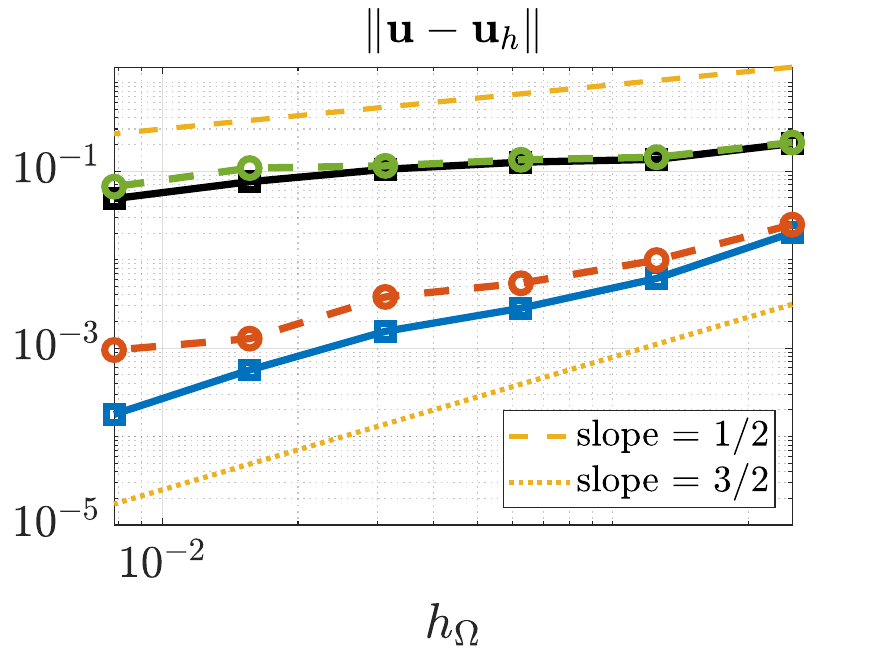}\\\medskip
		\includegraphics[width=0.4\linewidth]{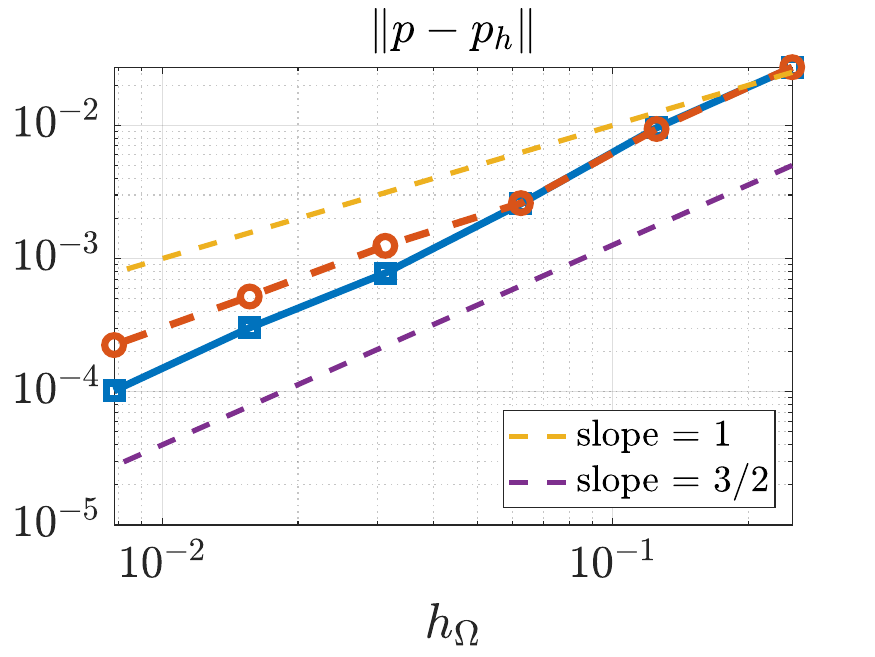}\qquad
		\includegraphics[width=0.4\linewidth]{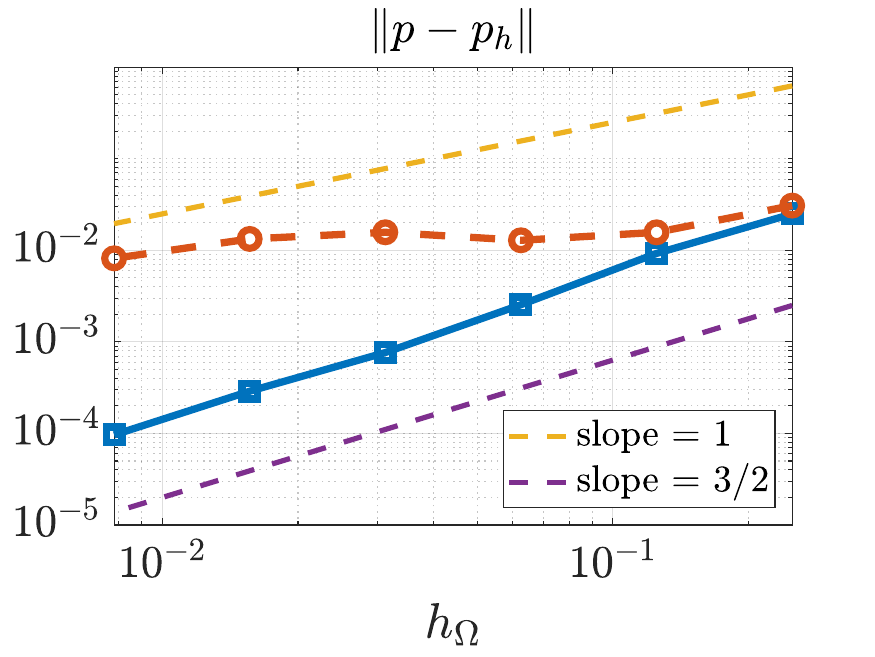}\\\medskip
		\includegraphics[width=0.4\linewidth]{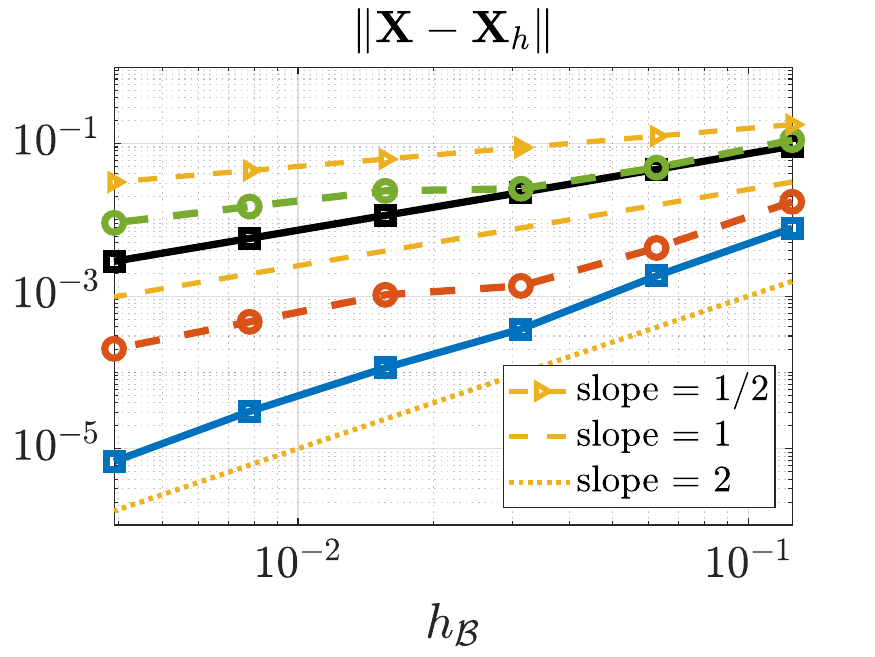}\qquad
		\includegraphics[width=0.4\linewidth]{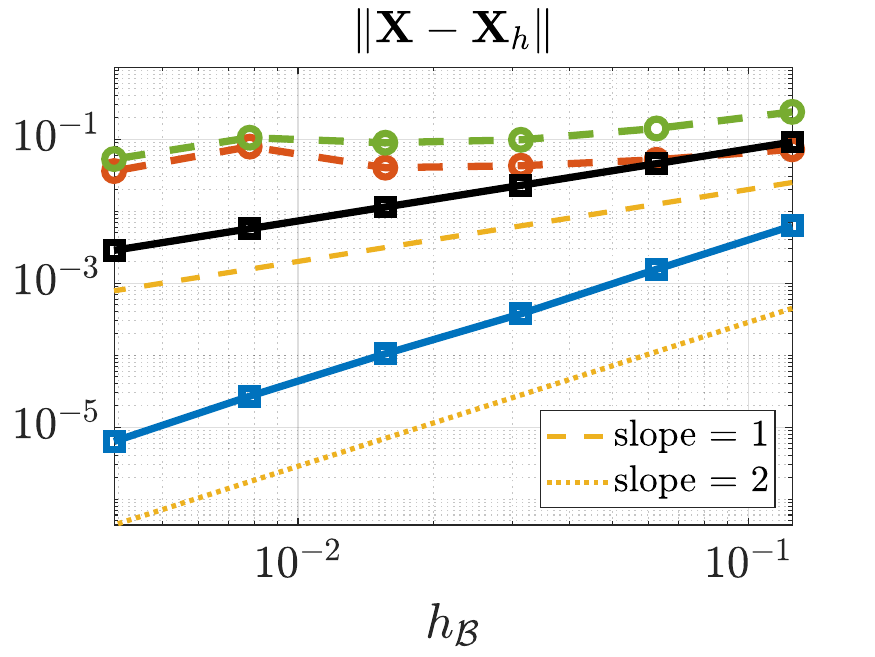}\\\medskip
		\includegraphics[width=0.4\linewidth]{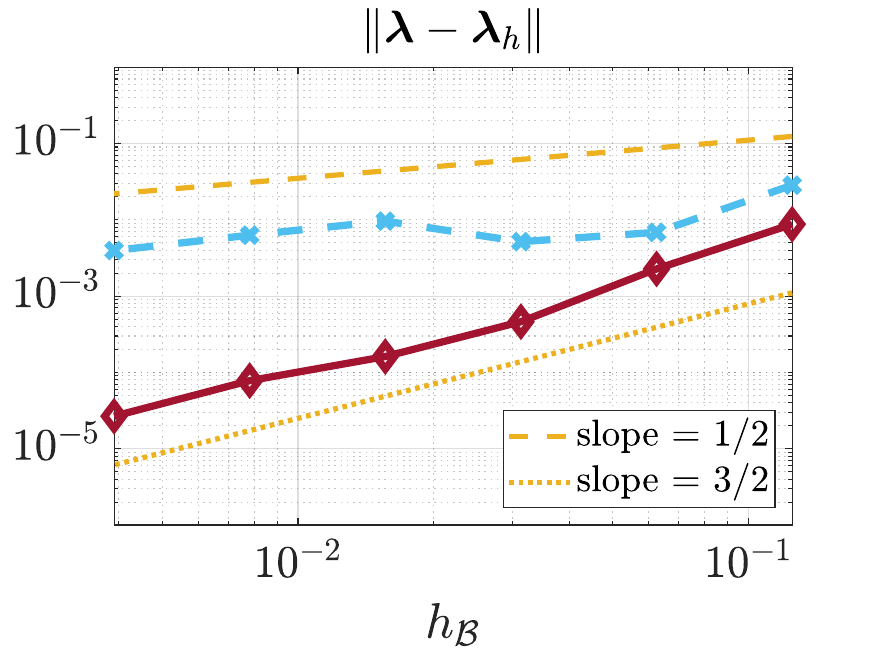}\qquad
		\includegraphics[width=0.4\linewidth]{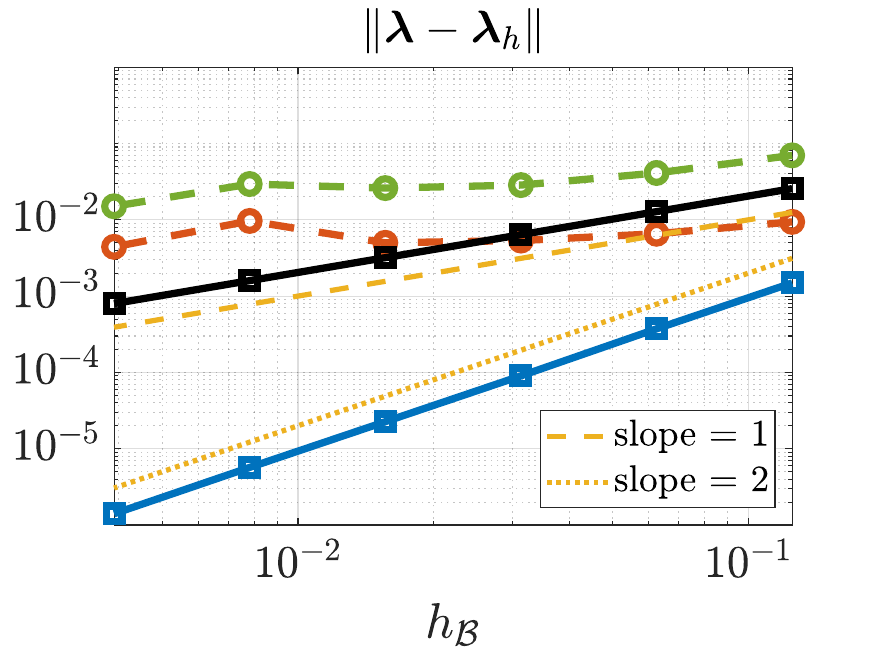}\\\medskip
		\includegraphics[width=1\linewidth]{figures/legend}\\\medskip
		\caption{\DB Convergence history of $\u,p,\X,\llambda$ in Test~3: comparison between exact and approximate assembly of the interface matrix. The results obtained with $\c=\c_0$ are collected in the left column, while the right column is related to $\c=\c_1$.\BD}
		\label{fig:test3_results}
	\end{figure}
	
	\subsubsection*{Test 4}
	We repeat Test~1 by considering higher order finite elements, even if this case is not covered by our theory. More precisely, we discretize the fluid unknowns by the Hood--Taylor $\Pcal_2/\Pcal_1$ finite element, while the solid unknowns are discretized by conforming piecewise quadratic elements. Hence, all variables have theoretical convergence rate equal to two. The coupling term is assembled by choosing a quadrature rule which is exact for polynomials of degree four. Convergence plots are collected in Figure~\ref{fig:test4_results}: the left column is related to $\c=\c_0$, while the right column refers to $\c=\c_1$. As already observed in the case of linear finite elements, the $\LdBd$ coupling term gives optimal convergence rates even when assembled by the inexact procedure. We observe that the error on $\llambda$ converges with rate $1/2$ when the inexact quadrature rule is employed. On the other hand, the $\Hub$ coupling term provides optimal results only when computed exactly.
	\BD

	\begin{figure}
		\centering
		\DB
		\textbf{Test 4: $h_\B\rightarrow0$\\$\LdBd$ coupling \textit{vs} $\Hub$ coupling }\\
		\vspace{3mm}\BD
		\includegraphics[width=0.4\linewidth]{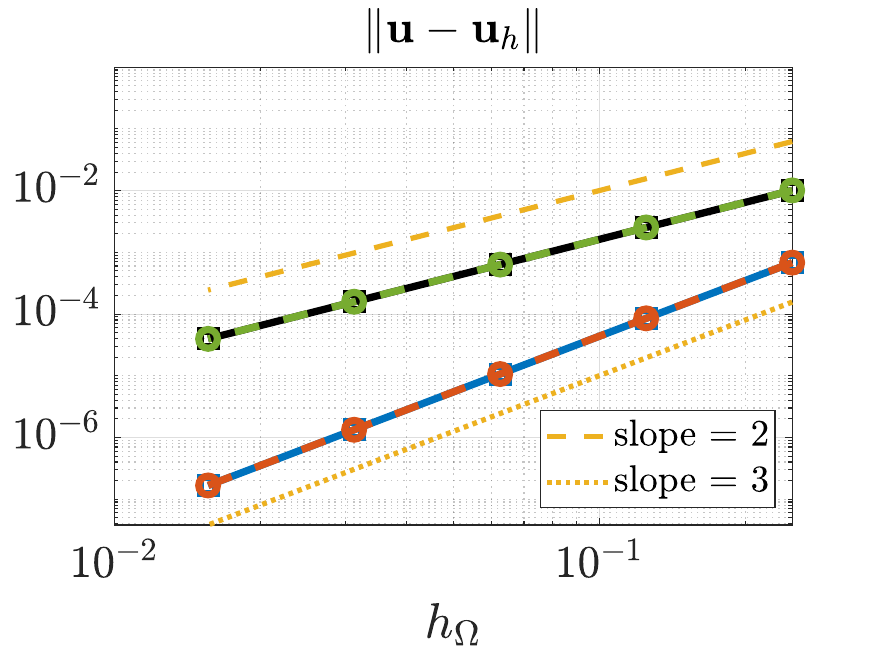}\qquad
		\includegraphics[width=0.4\linewidth]{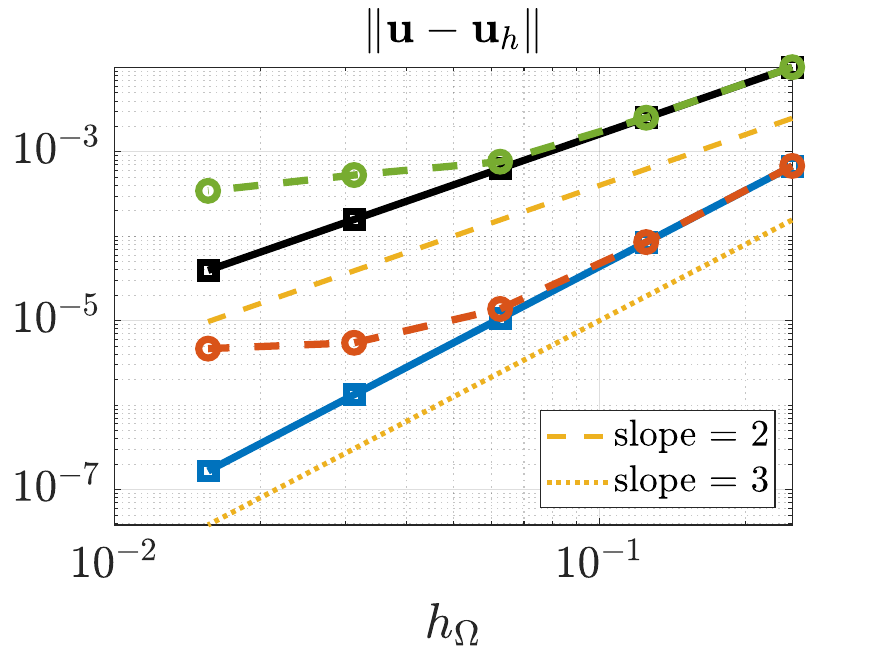}\\\medskip
		\includegraphics[width=0.4\linewidth]{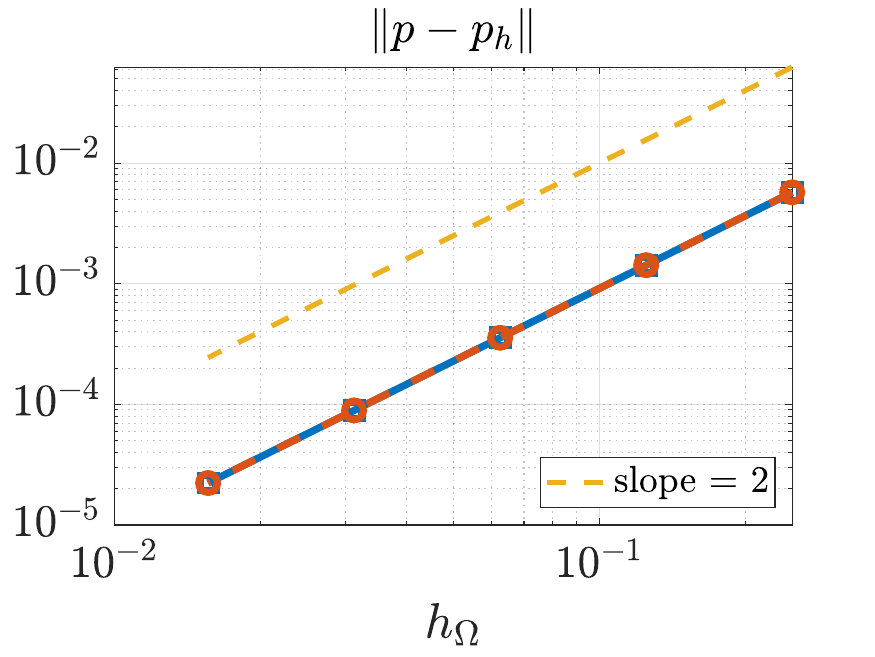}\qquad
		\includegraphics[width=0.4\linewidth]{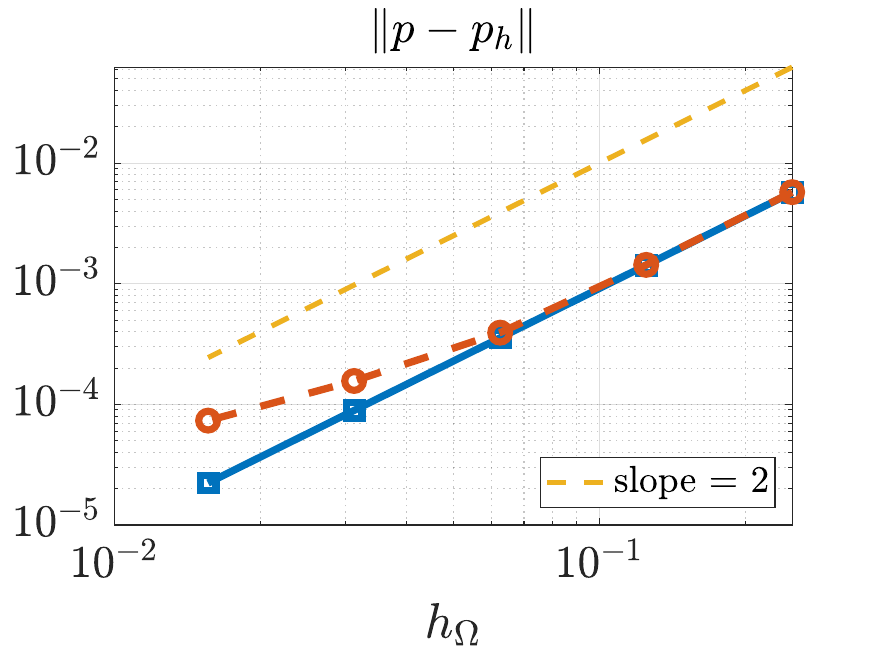}\\\medskip
		\includegraphics[width=0.4\linewidth]{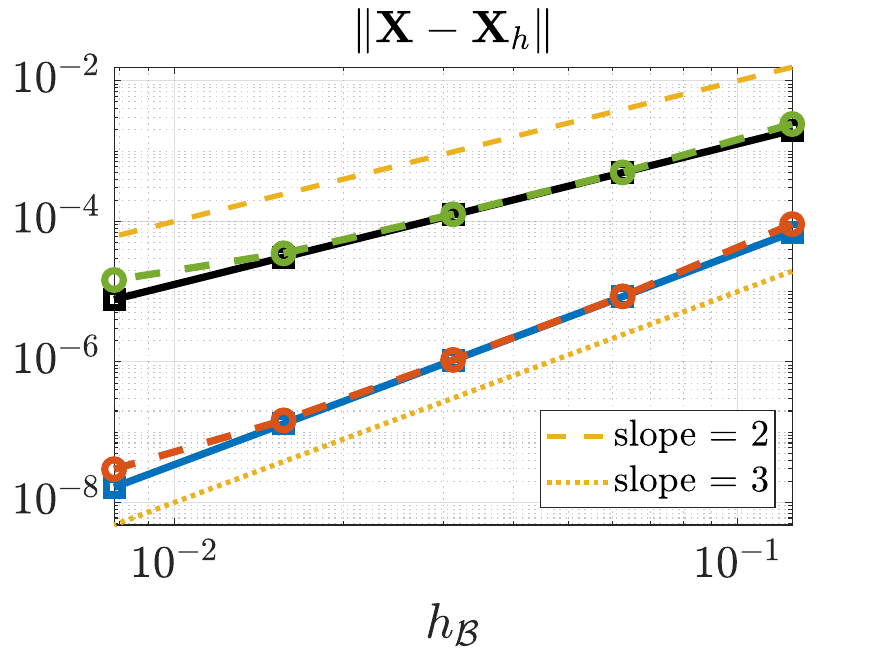}\qquad
		\includegraphics[width=0.4\linewidth]{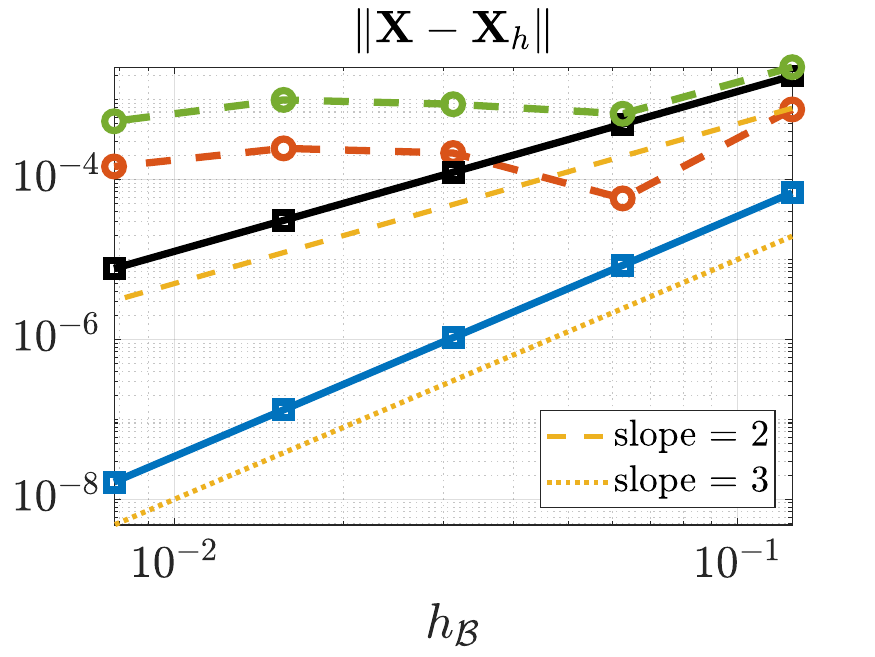}\\\medskip
		\includegraphics[width=0.4\linewidth]{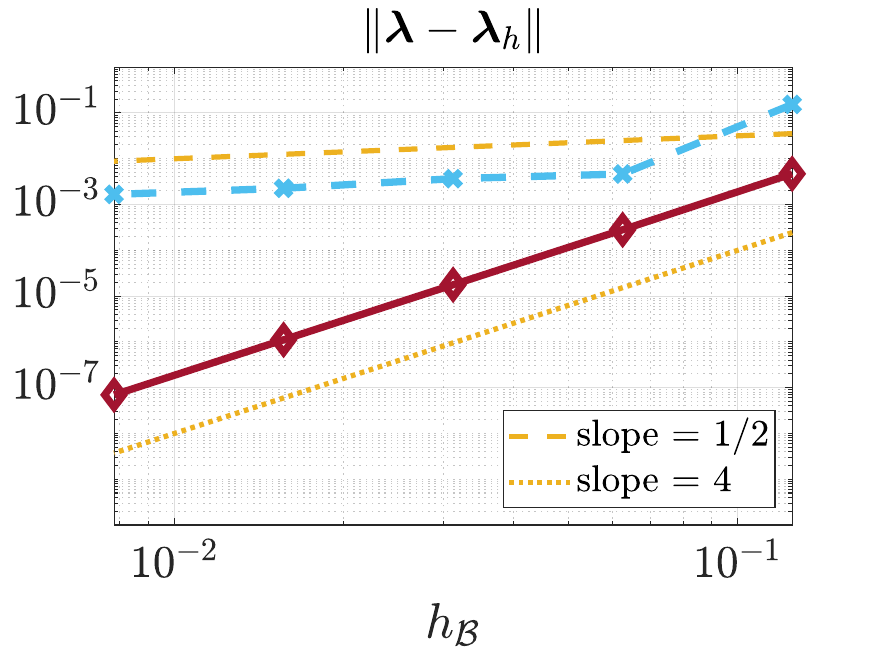}\qquad
		\includegraphics[width=0.4\linewidth]{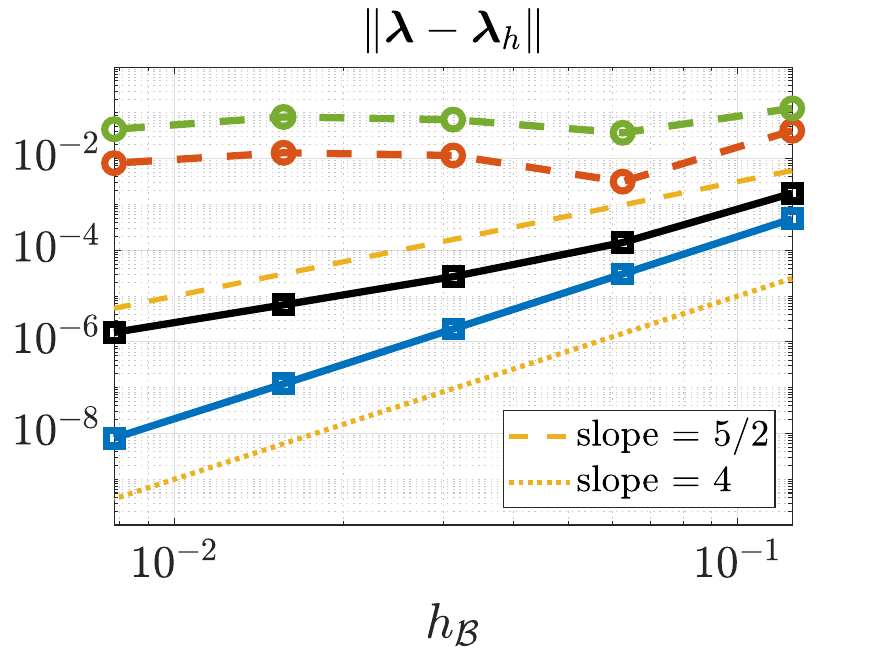}\\\medskip
		\includegraphics[width=1\linewidth]{figures/legend}\\\medskip
		\caption{\DB Convergence history of $\u,p,\X,\llambda$ in Test~4: comparison between exact and approximate assembly of the interface matrix. The results obtained with $\c=\c_0$ are collected in the left column, while the right column is related to $\c=\c_1$.\BD}
		\label{fig:test4_results}
	\end{figure}
	
	\begin{rem}
		The numerical investigations performed in our previous work
		\cite{boffi2022interface} were focused on the case of coupling with $\c=\c_1$.
		We remark that the results we obtained at that time are consistent with the
		quadrature error estimate proved in Proposition~\ref{prop:h1_final}. Indeed,
		all the tests where performed considering ${h_\B}/{h_\Omega}$ constant and
		showed a lack of optimality for the method constructed with the approximate
		coupling term. Proposition~\ref{prop:h1_final} requires that both ${h_\B}$ and
		${h_\B}/{h_\Omega}$ decrease to zero in order to obtain an optimal method without mesh intersection.
	\end{rem}
	
	\begin{rem}
		In \cite{boffi2022parallel}, we studied a first example of parallel solver for the fictitious domain formulation under consideration. 
		The coupling bilinear form was set to be $\c_0$. 
		The performance of the proposed solver has been assessed also by looking at the percentage of volume loss of the immersed solid between the last and the first time instant of simulation. We observed that the results were admissible also when the approximate coupling is considered; more precisely, the percentage of volume loss was smaller in the case without mesh intersection. The good performance is now justified by the quadrature error estimate obtained in Proposition~\ref{prop:l2_final}.
	\end{rem}
	
	\section{Conclusions}
	
	We discussed how to deal with the coupling term in the finite element discretization of fluid-structure interaction problems modeled by the formulation with distributed Lagrange multiplier introduced in~\cite{2015}: this is used to enforce the kinematic condition in a fictitious domain framework.
	
	The construction of the coupling term consists in integrating on the solid domain both fluid and solid variables, which are defined on two independent non-matching grids. Two approaches can be considered: the \textit{exact approach} requires the computation of the intersection between the two meshes, so that a composite quadrature rule is implemented, whereas the \textit{approximate approach} consists in inexact integration over the solid domain. We focused on the quadrature error due to the effect of inexact integration.
	
	We showed that the discrete problem is stable, i.e. it satisfies the
	inf-sup conditions, also when the coupling term is computed inexactly.
	Moreover, we proved quadrature error estimates: the coupling form $\c_0$
	behaves well provided that the solid mesh size $h_\B\rightarrow0$, whereas for $\c_1$ also the condition $h_\B/h_\Omega\rightarrow0$ is required. Our theoretical results are validated by two numerical tests and are also consistent with the experimental results we obtained in our previous studies \cite{boffi2022interface,boffi2022parallel}.
	
	\section*{Acknowledgments}
	The authors are members of INdAM Research group GNCS. The research of L. Gastaldi is partially supported by PRIN/MUR (grant No.20227K44ME). D. Boffi and L. Gastaldi are partially supported by IMATI/CNR.

	\bibliographystyle{abbrv}
	\bibliography{biblio}
\end{document}